\documentclass[review]{elsarticle}
\usepackage[usenames]{color}
\usepackage{lineno,hyperref}

\journal{Journal of \LaTeX\ Templates}

\usepackage{geometry}
\usepackage{amsfonts}
\usepackage{amsmath, cases}
\usepackage{indentfirst}
\usepackage{graphicx}
\usepackage{latexsym,bm, amsthm}
\usepackage{epsfig}
\usepackage{latexsym}
\usepackage{amssymb}\numberwithin{equation}{section}
\newtheorem{theorem}{T{\scriptsize HEOREM}}[section]

\newtheorem{definition}{D{\scriptsize  EFINITION}}[section]

\newtheorem{example}{E{\scriptsize  XAMPLE}}[section]

\begin{document}

\begin{frontmatter}
\title{Characterizations and properties of  dual matrix star  orders}

\author[mymainaddress]{Hongxing Wang}
\ead{winghongxing0902@163.com}

\author[mymainaddress]{Pei Huang}
\ead{Huangpei199904@163.com}

\address[mymainaddress]{College of Mathematics and Physics,
Guangxi Key Laboratory of Hybrid Computation and IC Design Analysis,
Guangxi Minzu University,
Nanning 530006, China}

\begin{abstract}
In this paper, we  introduce  D-star order, T-star order and  P-star order on the class of dual matrices. By applying matrix decomposition and dual generalized inverses, we  discuss  properties, characterizations and relations among these orders, and illustrate their relations  with examples.
\end{abstract}

\end{frontmatter}


\section{Introduction}
\newcommand{\rk}{\rm{rk}}
\newcommand{\Ind}{{\rm{Ind}}}
\newcommand{\CM}{\tiny\mbox{\rm{CM}}}

In this paper, we adopt the following notations.
The symbol $\mathbb{R}^{m\times n}$
 denotes the set of all $m\times n$ real matrices.
$A^T$  and ${\rk}(A)$
denote the
transpose and rank of $A\in\mathbb{R}^{m\times n}$, respectively.
The Moore-Penrose inverse of $A\in\mathbb{R}^{m\times n}$ is
defined as the unique matrix $X\in\mathbb{R}^{n\times m}$
satisfying the Penrose equations:
$ AXA=A$, $XAX=X$,
$\left(AX\right)^\ast=AX$ and $\left(XA\right)^\ast=XA$,
and is usually denoted by $X=A^{\dag}$ \cite{R. Penrose 1955}.
Denote an $m\times n$  dual matrix by $\widehat  A=A+\varepsilon A_0$,
in which
$A$ and $A_0$ are all $m\times n$ real matrices,
and
$\varepsilon$ is the dual unit satisfying
$\varepsilon\neq0,\;0\varepsilon=\varepsilon0,\;
1\varepsilon=\varepsilon1=\varepsilon$,
$\varepsilon^2=0$.
Furthermore,
$\widehat A^{ T}$ denotes the  transpose of $\widehat  A$, that is, $A^{ T}=A^{ T}+\varepsilon A_0^{ T}$.
$\mathbb{D}^{m\times n}$ denotes  the set of all $m\times n$ dual matrices.

Dual matrices have been commonly used in various fields of science and engineering,
such as the kinematic analysis  synthesis of machines and mechanisms, robotics
and machine vision\cite{Udwadia F E.X 2021}.
Recently,  dual generalized inverses attracted much attention.
 Many researchers have acquired fruitful findings
 \cite{Belzile B X2019,Belzile B2020,Stefanelli R E2007,Valentini P P 2018}.
Let $\widehat  A=A+\varepsilon A_0\in\mathbb{D}^{m\times n}$, then
the   Moore-Penrose dual generalized inverse (MPDGI for short)
of  $\widehat A$ is denoted by $\widehat  A^{p}$  \cite{Stefanelli R E2007} and  displayed in the form
\begin{align}
\label{MPDGI-Def}
\widehat  A^{p}=A^{\dag}-\varepsilon A^{\dag}A_0A^{\dag}.
\end{align}
Obviously, every  dual matrix has MPDGI.
Pennestr\`{\i} et al.\cite{Valentini P P 2018}
 propose  novel and computationally efficient algorithms(formulas) for the
computation of the MPDGI.

If there exists a unique matrix $\widehat X\in\mathbb{D}^{n\times m}$ satisfying the Penrose equations:
\begin{align}
\label{DMPGI-Def}
\left(\widehat{1}\right)\;\widehat A\widehat X\widehat A=\widehat A, \
\left(\widehat{2}\right)\;\widehat X\widehat A\widehat X=\widehat X,\
\left(\widehat{3}\right)\;\widehat A\widehat X=\left(\widehat A\widehat X\right)^{  T},\
\left(\widehat{4}\right)\;\widehat X\widehat A=\left(\widehat X\widehat A\right)^{  T}
\end{align}
then $\widehat X$ is the dual Moore-Penrose generalized inverse (DMPGI for short)
 of $\widehat A$ \cite{Udwadia F E2020},
and denoted  by $\widehat X=\widehat A^{\dag}$.
Udwadia \cite{Udwadia F E2020}
shows  that not all  dual
matrices  have DMPGIs,
and gets some interesting properties of   DMPGI.
 Wang \cite{Wang H.2021} gives a compact formula for DMPGI. He also puts forward some  necessary and sufficient conditions for a
dual matrix to have  DMPGI.
These theories should be main tools to carry out studies on dual matrix partial order  in this paper.

Dual generalized inverses is a powerful tool to study the 
least-squares solutions to systems of linear dual equations \cite{Udwadia F E.2021}.
For example,
Belzile \cite{Belzile B X2019,Belzile B2020}
uses  dual generalized inverses,
the characteristic length and Householder reflections over the dual ring
to investigate  problems of both translation
and rotation in the realm of kinematic synthesis.
These applications provide the impetus for the in-depth study on  dual generalized inverse theory.

It is well known that an important application of generalized inverse is to study  matrix partial order theory,
such as   characterizations  and representations   of
star, sharp, core   and minus
partial orders \cite{Baksalary O M 2010,Hartwig R E1980,Mitra S K 2010}.
Matrix partial order theory can be applied
to solve optimization problems like the minimization of production costs in statistics \cite{Golubic I 2020}.
The theory is also used to study
 autonomous linear systems and control system problems \cite{Coll C2020,Herrero A2020}.

Abundant theories of dual generalized inverses provide
a sufficient basis for carrying out researchs on
dual matrix order theory and practice.
Because the dual matrix structure is special,
it makes DMPGI and MPDGI are closely related but they are different in essence.
These differences provide a basis
for conducting studies on dual matrix partial order
to obtain rich and interesting  results.
The expectant research results of dual matrix order will be more diversified.
For example,
we can use the transpose of real  matrices or Moore-Penrose inverse to characterize star partial order,
 but we cannot get similar results in the dual matrix partial order.
Since the existence of DMPGI has strict conditions,
but MPDGI always exists.
Therefore, both of the dual binary relations are not equivalent.
Next,
we will investigate star order of dual matrices.
The theoretical results will also provide a theoretical basis for linear systems of  dual equations.

The outline of this paper is as follows.
In Section \ref{Preliminaries},
we briefly review some preliminaries.
In Section \ref{Sect-3-D-Star-Partial-Order},
when DMPGIs of dual matrices exist,
we introduce the D-star order of dual matrices,
and give some necessary and sufficient conditions for the existence of D-star order.
Furthermore,we
  prove that it is a partial order
  and
derive characterizations and properties of
the   partial order
by applying matrix decomposition.
In Section \ref{Sect-4-P-Star-Partial-Order},
we present
a new binary relation(P-order)
by applying  MPDGIs.
When  DMPGIs of dual matrices  exist,
it is shown that the new binary relation is partial order
is called P-star  partial order.
In Section
\ref{Sect-5-Relationships-Star-Partial-Order},
we consider relations between
D-star partial order and P-star  partial order,
 and
  give examples to illustrate their differences and connections.

\section{Preliminaries}
\label{Preliminaries}
In this section, we give some basic theories for further research,
such as
the singular value decomposition (SVD for short) of real matrix,
 characterizations of star partial order and DMPGI and so on.

\begin{theorem}
[SVD]
\label{SVD}
Let $A \in \mathbb{R}^{m\times n}$ and ${\rk}\left( A \right)=a$.
Then
 there exist orthogonal matrices
$U \in \mathbb{R}^{m\times m}$
and
$V \in \mathbb{R}^{n\times n}$
such that
\begin{align}
A=U\left(\begin{matrix}
   T_1    &0
\\ 0         &0
\end{matrix}\right) V^T , \nonumber
\end{align}
where $T_1 \in \mathbb{R}^{a\times a}$  is a diagonal positive definite matrix.
\end{theorem}

\begin{theorem}[\cite{Baksalary J K2003,Mitra S K 2010}]
\label{RStar-Def}
Let $A,B\in\mathbb{R}^{m\times n}$,
${\rk}\left( A \right)=a$,
 ${\rk}\left( B \right)=b$
 and
 $b> a$.
Then the following four statements are equivalent:
\begin{enumerate}
  \item[{\rm (1)}]
$A\overset\ast\leqslant B$;

  \item[{\rm (2)}]
$A^\dagger A=A^\dagger B$
and
$ AA^\dagger=BA^\dagger $;

  \item[{\rm (3)}]
 $A^{  T} A=A^{  T} B$
and
$AA^{  T}=BA^{  T}$;

  \item[{\rm (4)}]
There exist orthogonal matrices $U$ and $V$ such that
\begin{align}
\label{2.1}
A=U\begin{pmatrix}T_1&0&0\\0&0&0\\0&0&0\end{pmatrix}V^T, \
B=U\begin{pmatrix}T_1&0&0\\0&T_2&0\\0&0&0\end{pmatrix}V^T,
\end{align}
where
$T_1 \in \mathbb{R}^{a\times a}$
and
$T_2 \in \mathbb{R}^{(b-a)\times (b-a)}$
are diagonal positive definite matrices.
\end{enumerate}
\end{theorem}

\begin{theorem}
[\cite{Wang H.2021} Theorem 2.1]
\label{lemma-2.1}
Let $\widehat A=A+\varepsilon A_0 \in \mathbb{D}^{m\times n}$.
Then the following conditions are equivalent:
\begin{description}
  \item$\left(1\right)$
  The DMPGI $\widehat A^{\dag}$ of $\widehat A$ exists;

  \item$\left(2\right)$
  $\left( I_m-AA^{\dag} \right) A_0 \left( I_n-A^{\dag}A \right)=0$;

  \item$\left(3\right)$
  \
  ${\rk}\left(\begin{matrix}
  A_0  &A\\ A  &0
  \end{matrix}\right)
  =2{\rk}\left( A \right)$.
\end{description}
If the DMPGI $\widehat A^{\dag}$ of $\widehat A$ exists, then
\begin{align}
\label{eq-2.1}
\widehat A ^{\dag}=A ^{\dag}+\varepsilon R,
\end{align}
where
$R=-A^{\dag} A_0 A^{\dag}+\left( A^T A \right)^{\dag} A_0^T \left( I_m-A A^{\dag} \right)
  + \left( I_n-A^{\dag} A \right) A_0^T \left( A A^T \right)^{\dag}.$

Furthermore,
let the SVD of $A$ be as shown in Theorem \ref{SVD},
then
\begin{align}
\label{2.2}
\widehat A
&
=
U\begin{pmatrix}T_1&0\\0&0\end{pmatrix}V^{  T}+\varepsilon U\begin{pmatrix}A_1&A_2\\A_3&0\end{pmatrix}V^{  T},
\\
\widehat  A^{\dag}
&
=
V\begin{pmatrix}T_1^{-1}&0\\0&0\end{pmatrix} U^{  T}
+\varepsilon V\begin{pmatrix}
-T_1^{-1}A_1T_1^{-1}&T_1^{-2}A_3^{  T}\\
A_2^{  T}T_1^{-2}&0
\end{pmatrix}U^{  T},
\end{align}
where $T_1$ is a diagonal positive definite matrix.
\end{theorem}

\begin{theorem}
[\cite{Wang H.2021}]
\label{MPDGI-Cha}
Let $\widehat  A=A+\varepsilon A_0$.
Then MPDGI of $\widehat  A$, i.e. ${\widehat A}^{p}$ always exists,
and there exist orthogonal matrices $U$ and $V$ such that
\begin{align}
\widehat A
&
=
U\begin{pmatrix}T_1&0\\0&0\end{pmatrix}V^{  T}+\varepsilon U\begin{pmatrix}A_1&A_2\\A_3&A_4\end{pmatrix}V^{  T},
\\
\widehat A^p
&
=
V\begin{pmatrix}T_1^{-1}&0\\0&0\end{pmatrix}  U^{  T}
+
\varepsilon V
\begin{pmatrix}
-T_1^{-1}A_1T_1^{-1}&0\\
0&0
\end{pmatrix}  U^{  T},
\end{align}
where $T_1$ is a diagonal positive definite matrix.
\end{theorem}

\section{D-Star Partial Order}
\label{Sect-3-D-Star-Partial-Order}

The  DMPGI satisfies  Penrose equations
and
is closely related to Moore-Penrose generalized inverse of  real matrix.
Therefore,
we firstly introduce the D-star order by applying DMPGI.

\begin{definition}
\label{DStarPartialOrder-Def}
Let
DMPGIs of $\widehat  A$ and $\widehat  B$   exist.
If $\widehat  A$, $\widehat  B$ satisfy
\begin{align}
\label{DStarPartialOrder}
\widehat A^{\dag}\widehat A=\widehat A^{\dag}\widehat  B
\ \mbox{ and }\
\widehat A\widehat A^{\dag}=\widehat B\widehat A^{\dag},
\end{align}
we say that $\widehat  A$ is below $\widehat  B$ under the   D-star order,
and denote it by
  $\widehat A\overset{\tiny\mbox{\rm  D\!-}\ast}\leq\widehat B$.
\end{definition}

\begin{theorem}
\label{DSPO-Char-1-Th}
Let $\widehat  A=A+\varepsilon A_0$
and
$\widehat  B=B+\varepsilon B_0$,
where $A$, $A_0$, $B$ and $B_0\in\mathbb{R}^{m\times n}$.
And let  DMPGIs of $\widehat  A$ and $\widehat  B$  exist,
then
$\widehat A\overset{\tiny\mbox{\rm  D\!-}\ast}\leq\widehat B$
if and only if
\begin{align}
\label{DSPO-Char-1}
\left\{\begin{array}{l}
A\overset\ast\leq B  \\
A^{\dag}A_0+RA=A^{\dag}B_0+RB  \\
AR+A_0A^{\dag}=BR+B_0A^{\dag},
\end{array}\right.
\end{align}
where
$ R=-A^\dagger A_0A^\dagger
+\left(A^{  T}A\right)^\dagger A_0^{  T}\left(I_m-AA^\dagger\right)
+\left(I_n-A^\dagger A\right)A_0^{  T}\left(AA^{  T}\right)^\dagger$.
\end{theorem}

\begin{proof}
Let
$\widehat  A=A+\varepsilon A_0$,
$\widehat  B=B+\varepsilon B_0$,
and
 DMPGI $\widehat  A^{\dag}$ of $\widehat  A$ exist.
 Denote
 $\widehat  A^{\dag}=A^{\dag}+\varepsilon R$,
 where $ R$
 is  as in  Theorem \ref{lemma-2.1}.
Then
\begin{align}
\nonumber
\left\{\begin{array}{l}\widehat A^\dagger\widehat A
=
\left(A^\dagger+\varepsilon R\right)\left(A+\varepsilon A_0\right)
=A^\dagger A+\varepsilon\left(A^\dagger A_0+RA\right)
\\
\widehat A^\dagger\widehat B
=\left(A^\dagger+\varepsilon R\right)\left(B+\varepsilon B_0\right)
=A^\dagger B+\varepsilon\left(A^\dagger B_0+RB\right)
\end{array}\right.
\end{align}
and
\begin{align}
\nonumber
\left\{\begin{array}{l}
\widehat A\widehat A^\dagger
=\left(A+\varepsilon A_0\right)\left(A^\dagger+\varepsilon R\right)
=AA^\dagger+\varepsilon\left(AR+A_0A^\dagger\right)
\\
\widehat B\widehat A^\dagger
=\left(B+\varepsilon B_0\right)\left(A^\dagger+\varepsilon R\right)
=BA^\dagger+\varepsilon\left(BR+B_0A^\dagger\right).
\end{array}\right.
\end{align}

Since $\widehat A\overset{\tiny\mbox{\rm  D\!-}\ast}\leq\widehat B$,
it follows from Definition \ref{DStarPartialOrder-Def}
 that $\widehat A\overset{\tiny\mbox{\rm  D\!-}\ast}\leq\widehat B$
if and only if
\begin{align}
\nonumber
\left\{\begin{array}{l}
 A^\dagger A=A^\dagger B,\;AA^\dagger=BA^\dagger \\
 A^\dagger A_0+RA=A^\dagger B_0+RB \\
 AR+A_0A^\dagger=BR+B_0A^\dagger.
\end{array}\right.
\end{align}

Since
$A^\dagger A=A^\dagger B$
and
$AA^\dagger=BA^\dagger$,
we get
 $A\overset\ast\leq B$.
Therefore,
 $\widehat A\overset{\tiny\mbox{\rm  D\!-}\ast}\leq\widehat B$
 is equivalent to (\ref{DSPO-Char-1}).
\end{proof}

\begin{theorem}
\label{DSPO-Char-2-Th}
Let $\widehat  A=A+\varepsilon A_0\in\mathbb{D}^{m\times n}$,
$\widehat  B=B+\varepsilon B_0\in\mathbb{D}^{m\times n}$,
and  DMPGIs of $\widehat  A$ and $\widehat  B$   exist.
Then
$\widehat A\overset{\tiny\mbox{\rm  D\!-}\ast}\leq\widehat B$
 if and only if
there exist orthogonal matrices $U$ and $V$
{\small
\begin{align}
\label{DSPO-Char-2}
\left\{\begin{array}{l}\widehat A
=U\begin{pmatrix}
 T_1&0&0\\
0&0&0\\
0&0&0\end{pmatrix}V^{  T}
+\varepsilon U
\begin{pmatrix}
A_1&A_2&A_3\\
A_4&0&0\\
A_7&0&0
\end{pmatrix}V^{  T},
\\
\widehat B=
U\begin{pmatrix} T_1&0&0\\
0& T_2&0\\
0&0&0
\end{pmatrix}
V^{  T}
+\varepsilon U\begin{pmatrix}A_1
&A_2-{ T_1^{-1}}{  A_4^{  T}T_2}&A_3\\
A_4-{ T_2}{  A_2^{  T}}{ T_1^{-1}}&
  B_5&B_6\\A_7&B_8&0
  \end{pmatrix}V^{  T},
  \end{array}\right.
\end{align}}
where $T_1$ and $T_2$ are diagonal positive definite matrices.
\end{theorem}
\begin{proof}
``$\Rightarrow$"
\quad
Denote
${\rk}\left( A \right)=a$
and
 ${\rk}\left( B \right)=b$.
Since $\widehat A\overset{\tiny\mbox{\rm  D\!-}\ast}\leq\widehat B$,
by applying Theorem \ref{DSPO-Char-1-Th},
we get $A\overset\ast\leq B$.
Then
$A$ and $B$ are of the forms as in (\ref{2.1}).
Since  the DMPGI  of $\widehat  A$  exists,
we write
\begin{align}
\label{3.5}
A_0=
U\begin{pmatrix}
\textstyle A_1&\textstyle A_2&\textstyle A_3\\
\textstyle A_4&0&0\\
\textstyle A_7&0&0
\end{pmatrix}V^{  T},
\end{align}
where $A_1\in\mathbb{R}^{a\times a}$,
   $A_2\in\mathbb{R}^{a\times (b-a)}$
 and $A_4\in\mathbb{R}^{(b-a)\times a}$.
Applying Theorem \ref{lemma-2.1}, we have
\begin{align}
\label{3.6}
\widehat A^\dagger=V\begin{pmatrix}
 T_1^{-1}&0&0
\\0&0&0
\\0&0&0
\end{pmatrix}U^{  T}
+
\varepsilon V
\begin{pmatrix}
-T_1^{-1}A_1T_1^{-1}&T_1^{-2}A_4^{  T}& T_1^{-2}A_7^{  T}\\
A_2^{  T}T_1^{-2}&0&0\\
A_3^{  T}T_1^{-2}&0&0\end{pmatrix}U^{  T} .
\end{align}

Since the DMPGI of $\widehat B$ exists,
we  write
\begin{align}
\label{3.7}
B_0=U\begin{pmatrix}  B_1&  B_2&  B_3\\
  B_4&  B_5&  B_6\\
  B_7&  B_8&0\end{pmatrix}V^{  T},
\end{align}
where $B_1\in\mathbb{R}^{a\times a}$ and $B_2\in\mathbb{R}^{a\times (b-a)}$.
Applying (\ref{2.1}), (\ref{3.5}), (\ref{3.6}) and (\ref{3.7}) gives
{\small
\begin{align}
\label{3-3-1}
\left\{
\begin{aligned}
\widehat A^\dagger\widehat A
&
=V\begin{pmatrix}
I&0&0\\0&0&0\\0&0&0\end{pmatrix}V^{  T}
+\varepsilon V\begin{pmatrix}   0&
 T_1^{-1}A_2& T_1^{-1}A_3\\
    A_2^{  T}T_1^{-1}&0&0\\
    A_3^{  T}T_1^{-1}&0&0\end{pmatrix}V^{  T}
\\
\widehat A^\dagger\widehat B
&
=
V\begin{pmatrix}    I&0&0\\0&0&0\\0&0&0\end{pmatrix}V^{  T}
+\varepsilon V\begin{pmatrix} T_1^{-1}B_1-T_1^{-1}A_1&
 T_1^{-1}B_2
+T_1^{-2}A_4^{  T}T_2& T_1^{-1}B_3\\
    A_2^{  T}T_1^{-1}&0&0\\
    A_3^{  T}T_1^{-1}&0&0\end{pmatrix}V^{  T}.
\end{aligned}
\right.
\end{align}}

Since $\widehat A\overset{\tiny\mbox{\rm  D\!-}\ast}\leq\widehat B$,
we have
 $\widehat A^\dagger\widehat A=\widehat A^\dagger\widehat B$.
Applying (\ref{3-3-1}) gives
\begin{align}
\nonumber
\left\{\begin{array}{l}
0= T_1^{-1}B_1
 -T_1^{-1}A_1 \\
T_1^{-1}A_2=
 T_1^{-1}B_2+T_1^{-2}A_4^{  T}T_2 \\
 T_1^{-1}B_3
= T_1^{-1}A_3.
\end{array}\right.
\end{align}
Therefore,
$    B_1=A_1$,
    $B_2=A_2-T_1^{-1}A_4^{  T}T_2$ and
    $B_3=A_3$.
It follows from (\ref{3.7}) that
\begin{align}
\label{3.8}
B_0=U\begin{pmatrix}  A_1&    A_2-T_1^{-1}A_4^{  T}T_2&    A_3\\
B_4&B_5&B_6\\ B_7&B_8&0\end{pmatrix}V^{  T}.
\end{align}

Applying (\ref{2.1}), (\ref{3.5}), (\ref{3.6})  and (\ref{3.8}),
we obtain
{\small
\begin{align}
\label{3-3-2}
\left\{
\begin{aligned}
\widehat A\widehat A^\dagger
&
=U\begin{pmatrix}
  I&0&0\\0&0&0\\0&0&0\end{pmatrix}U^{  T}+
\varepsilon U\begin{pmatrix}
0&
 T_1^{-1}A_4^{  T}&
 T_1^{-1}A_7^{  T}\\
  A_4T_1^{-1}&0&0\\
  A_7T_1^{-1}&0&0\end{pmatrix}U^{  T}
\\
\widehat B\widehat A^\dagger
&
=
U\begin{pmatrix}  I&0&0\\0&0&0\\0&0&0\end{pmatrix}U^{  T}
+\varepsilon U\begin{pmatrix}
 0& T_1^{-1}A_4^{  T}& T_1^{-1}A_7^{  T}\\
  B_4T_1^{-1}+T_2A_2^{  T}T_1^{-2}&0&0\\
  B_7T_1^{-1}&0&0\end{pmatrix}U^{  T}.
\end{aligned}
\right.
\end{align}}
Since $\widehat A\overset{\tiny\mbox{\rm  D\!-}\ast}\leq\widehat B$,
we get
$\widehat A\widehat A^\dagger=\widehat B\widehat A^\dagger$.
It follows from  (\ref{3-3-2}) that
$A_4T_1^{-1}=B_4T_1^{-1}+T_2A_2^{  T}T_1^{-2}$
and
$A_7T_1^{-1}=B_7T_1^{-1}$,
that is,
\begin{align}
\label{3.10}
B_4=A_4-T_2A_2^{  T}T_1^{-1}
\mbox{ \ and } \
B_7=A_7.
\end{align}
Therefore,
applying (\ref{2.1}),
(\ref{3.8}) and (\ref{3.10}),
we get
\begin{align}
\nonumber
\widehat B
=
U\begin{pmatrix}
 T_1&0&0\\
0& T_2&0\\
0&0&0\end{pmatrix}V^{  T}
+\varepsilon U\begin{pmatrix}
A_1&A_2-{ T_1^{-1}}{  A_4^{  T}T_2}&A_3
\\
A_4-{ T_2}{    A_2^{  T}}{ T_1^{-1}}&  B_5&B_6
\\
A_7&B_8&0
\end{pmatrix}V^{  T}.
\end{align}

``$\Leftarrow$"
\quad
Let there exist orthogonal matrices $U$ and $V$
such that
$\widehat A$ and $\widehat B$ can be represented as  (\ref{DSPO-Char-2}).
Then the form of $\widehat A^\dagger$ is as in (\ref{3.6}).
It is easy to check that
\begin{align*}
\left\{
\begin{aligned}
\widehat A^\dagger\widehat A
&
=V_1\begin{pmatrix}
I&0&0\\0&0&0\\0&0&0\end{pmatrix}V_1^{  T}
+\varepsilon V_1\begin{pmatrix}{0}&
 T_1^{-1}A_2& T_1^{-1}A_3\\
  A_2^{  T}T_1^{-1}&0&0\\
  A_3^{  T}T_1^{-1}&0&0\end{pmatrix}V_1^{  T}
=
\widehat A^\dagger\widehat B
\\
\widehat A\widehat A^\dagger
&
=U\begin{pmatrix}
    I&0&0\\0&0&0\\0&0&0\end{pmatrix}U^{  T}+
\varepsilon U\begin{pmatrix}{0}&
 T_1^{-1}A_4^{  T}&
 T_1^{-1}A_7^{  T}\\
    A_4T_1^{-1}&0&0\\
    A_7T_1^{-1}&0&0\end{pmatrix}U^{  T}=
\widehat B\widehat A^\dagger .
\end{aligned}
\right.
\end{align*}
Therefore,
  applying Definition \ref{DStarPartialOrder-Def}
we get $\widehat A\overset{\tiny\mbox{\rm  D\!-}\ast}\leq\widehat B$.
\end{proof}

\begin{theorem}
\label{DSPO-Char-Partial-Th}
The   D-star order is a partial order.
\end{theorem}

\begin{proof}
Let $\widehat  A=A+\varepsilon A_0$;
$\widehat  B=B+\varepsilon B_0$;
 DMPGIs of
$\widehat  A$ and $ \widehat  B$   exist;
and
$\widehat A\overset{\tiny\mbox{\rm  D\!-}\ast}\leq\widehat B$;
  i.e.;
  $\widehat A^{\dag}\widehat A=\widehat A^{\dag}\widehat  B$
  and
  $\widehat A\widehat A^{\dag}=\widehat B\widehat A^{\dag}$.

Next, we show that the D-star order satisfies reflexivity,  anti-symmetry and transitivity.

(i) Reflexivity is self-evident.

(ii) Let
$\widehat  B\overset{\tiny\mbox{\rm  D\!-}\ast}\leq\widehat  A$.
Applying Theorem \ref{DSPO-Char-1-Th}, it follows from
$\widehat  A\overset{\tiny\mbox{\rm  D\!-}\ast}\leq\widehat  B$ that we have
 $   A\overset\ast\leq   B$ and
$   B\overset\ast\leq   A$.
From the anti-symmetry of star partial order on real matrices, we have
\  $A=B$.

Since $\widehat  A\overset{\tiny\mbox{\rm  D\!-}\ast}\leq\widehat  B$,
   $\widehat  A$ and $\widehat  B$ can be represented  in the forms as in  (\ref{DSPO-Char-2}).
Applying Theorem \ref{DSPO-Char-2-Th}, we have $T_2=0$.
Therefore,
{\small
\begin{align}
\label{DSPO-Char-Partial-1}
\left\{\begin{array}{l}\widehat A
=U\begin{pmatrix} T_1&0&0\\0&0&0\\0&0&0\end{pmatrix}V^{  T}
+\varepsilon U\begin{pmatrix}A_1&A_2&A_3\\A_4&0&0\\A_7&0&0\end{pmatrix}V^{  T}
\\
\widehat B=
U\begin{pmatrix} T_1&0&0\\0&0&0\\0&0&0\end{pmatrix}
V^{  T}
+\varepsilon U\begin{pmatrix}
A_1&A_2 &A_3
\\
A_4 &  B_5&B_6
\\
A_7&B_8&0
\end{pmatrix}V^{  T}.
\end{array}\right.
\end{align}}

Since the DMPGI of  $\widehat  B$ exists,
by applying Theorem \ref{lemma-2.1} and  (\ref{DSPO-Char-Partial-1}),
we get
$B_5=0$,
$B_6=0$
and
$B_8=0$.
Therefore,
 $\widehat  A= \widehat  B$.
So, the anti-symmetry holds.

(iii) 
Let $\widehat  C=C+\varepsilon C_0$;
 the DMPGI of $\widehat  C$ exist;
$\widehat B\overset{\tiny\mbox{\rm  D\!-}\ast}\leq\widehat C$.
Since
$\widehat A\overset{\tiny\mbox{\rm  D\!-}\ast}\leq\widehat B$
and
$\widehat B\overset{\tiny\mbox{\rm  D\!-}\ast}\leq\widehat C$, we get
$A \overset\ast\leq C$.
Denote
${\rk}\left( A \right)=a$,
 ${\rk}\left( B \right)=b$
and
 ${\rk}\left( C \right)=c$.
Then there exist
orthogonal matrices $U$ and $V$ such that
{\small
\begin{align}
\label{3-3-6}
A=U\begin{pmatrix}
 T_1&0&0&0\\
0&0&0&0\\
0&0&0&0\\
0&0&0&0
\end{pmatrix}V^{  T} , \
B=
 U\begin{pmatrix}
  T_1&0&0&0\\
0& T_2&0&0\\
0&0&0&0\\
0&0&0&0
\end{pmatrix}
V^{  T} , \
C=
 U\begin{pmatrix}
  T_1&0&0&0\\
0& T_2&0&0\\
0&0& T_3&0\\
0&0&0&0
\end{pmatrix}
V^{  T} ,
\end{align}}
where $T_1\in\mathbb{R}^{a\times a}$,
$T_2\in\mathbb{R}^{(b-a)\times (b-a)}$ and $T_3\in\mathbb{R}^{(c-b)\times (c-b)}$
are diagonal positive  definite  matrices.

Since $\widehat A\overset{\tiny\mbox{\rm  D\!-}\ast}\leq\widehat B$,
applying Theorem \ref{DSPO-Char-2-Th}, we get
{\small
\begin{align}
\label{3-3-4}
\left\{\begin{array}{l} A_0=
 U\begin{pmatrix}
A_1&A_2&A_{31}&A_{32}\\
A_4&0&0&0\\
A_{71}&0&0&0\\
A_{72}&0&0&0\end{pmatrix}
V^{  T}
\\
B_0=
 U\begin{pmatrix}
A_1&A_2-{ T_1^{-1}}{  A_4^{  T}T_2}&A_{31}&A_{32}\\
A_4-{ T_2}{  A_2^{  T}}{ T_1^{-1}}&
  B_5&B_{61}&B_{62}\\A_{71}&B_{81}&0&0\\
A_{72}&B_{82}&0&0\end{pmatrix}V^{  T}.
\end{array}\right.
\end{align}}
Since the DMPGI of $\widehat C$ exists,
we  denote
\begin{align}
\label{3-3-5}
C_0=U\begin{pmatrix}
  C_1&  C_2&  C_3&  C_4\\
  C_5&  C_6&  C_7&  C_8\\
  C_9&  C_{10}&  C_{11}&  C_{12}\\
  C_{13}&  C_{14}&  C_{15}&  0
\end{pmatrix}V^{  T}.
\end{align}
Applying Theorem \ref{lemma-2.1},
and  (\ref{3-3-6}) and (\ref{3-3-4}),
we get
{\small
\begin{align}
\label{3-3-8}
\widehat A^\dagger=V\begin{pmatrix}
 T_1^{-1}&0&0&0\\
0&0 &0&0\\
0&0&0&0\\
0&0&0&0\end{pmatrix}U^{  T}
+\varepsilon V\begin{pmatrix}
-T_1^{-1}A_1T_1^{-1}& T_1^{-2}A_4^{  T}&
 T_1^{-2}A_{71}^{  T}&
 T_1^{-2}A_{71}^{  T}\\
  A_2^{  T}T_1^{-2}& 0&0&0\\
  A_{31}^{  T}T_1^{-2}&0&0&0\\
  A_{32}^{  T}T_1^{-2}&0&0&0
\end{pmatrix}U^{  T}.
\end{align}}

Since $\widehat  B\overset{\tiny\mbox{\rm  D\!-}\ast}\leq\widehat  C$,
by
applying (\ref{3-3-6}), (\ref{3-3-4}), (\ref{3-3-5})
and Theorem \ref{DSPO-Char-2-Th},
we obtain
{\small
\begin{align}
\nonumber
\widehat C
&
=U\begin{pmatrix}
 T_1&0&0&0\\
0& T_2&0&0\\
0&0& T_3&0\\
0&0&0&0\end{pmatrix}V^{  T}
\\
\label{3-3-7}
&
\qquad
+\varepsilon
U\begin{pmatrix}
A_1&  A_2-T_1^{-1}A_4^{  T}T_2&A_{31}
-{ T_1^{-1}}{  A_{71}^{  T}}{ T_3}&
A_{32}\\
  A_4-T_2A_2^{  T}T_1^{-1}&
 B_5&
 B_{61}
-{ T_2^{-1}}{ B_{81}^{  T}}{ T_3}&
B_{62}\\
A_{71}-{ T_3}{  A_{31}^{  T}}{ T_1^{-1}}&
B_{81}-{ T_3}{  B_{61}^{  T}}{ T_2^{-1}}&
    C_{11}&C_{12}\\A_{72}&B_{82}&C_{15}&0
\end{pmatrix}V^{  T}.
\end{align}}

Then applying (\ref{3-3-8}),(\ref{3-3-7}) gives
{\small
\begin{align}
\nonumber
\left\{\begin{array}{l}
\widehat A^\dagger\widehat A
=V\begin{pmatrix}
  I&0&0&0\\
0& 0&0&0\\
0&0&0&0\\
0&0&0&0\end{pmatrix}V^{  T}+\varepsilon V\begin{pmatrix}
0& T_1^{-1}A_2& T_1^{-1}A_{31}&
 T_1^{-1}A_{32}\\
  A_2^{  T}T_1^{-1}&
 0&0&0\\
  A_{31}^{  T}T_1^{-1}&0&0&0\\
  A_{32}^{  T}T_1^{-1}&0&0&0
\end{pmatrix}V^{  T}
\\
\widehat A^\dagger\widehat C=
V\begin{pmatrix}
  I&0&0&0\\0&0&0&0\\0&0&0&0\\
0&0&0&0\end{pmatrix}V^{  T}+
\varepsilon V\begin{pmatrix}
0& T_1^{-1}A_2
& T_1^{-1}A_{31}&
 T_1^{-1}A_{32}\\
  A_2^{  T}T_1^{-1}&
 0&0&0\\
  A_{31}^{  T}T_1^{-1}&0&0&0\\
  A_{32}^{  T}T_1^{-1}&0&0&0
\end{pmatrix}V^{  T}
\end{array}\right.
\end{align}}
and
{\small
\begin{align}
\nonumber
\left\{\begin{array}{l}
\widehat A\widehat A^\dagger
=
U\begin{pmatrix}
  I&0&0&0\\
0&0&0&0\\
0&0&0&0\\
0&0&0&0
\end{pmatrix}U^{  T}+\varepsilon U
\begin{pmatrix}
0& T_1^{-1}A_4^{  T}&
 T_1^{-1}A_{71}^{  T}&
 T_1^{-1}A_{72}^{  T}\\
  A_4T_1^{-1}&0&0&0\\
  A_{71}T_1^{-1}&0&0&0\\
  A_{72}T_1^{-1}&0&0&0
\end{pmatrix}U^{  T}
\\
\widehat C\widehat A^\dagger
=
U\begin{pmatrix}
 I&0&0&0\\
0&0&0&0\\
0&0&0&0\\0&0&0&0
\end{pmatrix}U^{  T}+
\varepsilon U\begin{pmatrix}
0& T_1^{-1}A_4^{  T}&{T}_1^{-1}A_{71}^{  T}&
{T}_1^{-1}A_{72}^{  T}\\
  A_4{T}_1^{-1}&0&0&0\\
    A_{71}{T}_1^{-1}&0&0&0\\
    A_{72}{T}_1^{-1}&0&0&0
\end{pmatrix}U^{  T}.
\end{array}\right.
\end{align}}
It follows that
$
\widehat A^\dagger\widehat A=\widehat A^\dagger\widehat C$
and
$\widehat A\widehat A^\dagger=\widehat C\widehat A^\dagger$,
that is,
$\widehat A\overset{\tiny\mbox{\rm  D\!-}\ast}\leq\widehat C$.
Therefore,  the transitivity of  D-star order holds.
\end{proof}

\begin{theorem}
\label{DSPO-Char-4-Th}
Let
DMPGIs of
 $\widehat  A$ and $ \widehat  B$
  exist.
Then
$\widehat A\overset{\tiny\mbox{\rm  D\!-}\ast}\leq\widehat B$
if and only if
$\widehat A^\dagger \overset{\tiny\mbox{\rm  D\!-}\ast}\leq\widehat B^\dagger$.
\end{theorem}

\begin{proof}
$ \Rightarrow$
\quad
Let $\widehat  A=A+\varepsilon A_0$;
$\widehat  B=B+\varepsilon B_0$;
 DMPGIs of $\widehat  A $ and $\widehat  B$   exist;
 $\widehat A\overset{\tiny\mbox{\rm  D\!-}\ast}\leq\widehat B$.
Then
$\widehat  A$ and  $\widehat  B$
are of the forms as in  (\ref{DSPO-Char-2}).
So $\widehat  A^\dag$ can be represented in the form as in (\ref{3.6}),
and
{\small
\begin{align}
\label{3-3-9}
\widehat B^\dagger
=
V\begin{pmatrix}{T}_1^{-1}&0&0\\
0&{T}_2^{-1}&0\\
0&0&0\end{pmatrix}U^{  T}
+
\varepsilon V\begin{pmatrix}
-{T}_1^{-1}A_1{T}_1^{-1}&
{T}_1^{-2}A_4^{  T}-{T}_1^{-1}A_2{T}_2^{-1}&
{T}_1^{-2}A_7^{  T}\\
    A_2^{  T}{T}_1^{-2}-{T}_2^{-1}A_4{T}_1^{-1}&
-{T}_2^{-1}B_5{T}_2^{-1}&{T}_2^{-2}B_8^{  T}\\
    A_3^{  T}{T}_1^{-2}&B_6^{  T}{T}_2^{-2}&0
\end{pmatrix}U^{  T}.
\end{align}}

Since $\left(\widehat A^\dagger\right)^\dagger=\widehat A$,
by applying (\ref{DSPO-Char-2}),
 (\ref{3.6})  and (\ref{3-3-9}),
we get that
\begin{align}
\nonumber
\left(\widehat A^\dagger\right)^\dagger\widehat A^\dagger
=\left(\widehat A^\dagger\right)^{\dagger  }\widehat B^\dagger,
\ \
\widehat A^\dagger\left(\widehat A^\dagger\right)^{\dagger  }
=\widehat B^\dagger\left(\widehat A^\dagger\right)^{\dagger }.
\end{align}
It follows from Definition \ref{DStarPartialOrder-Def}
 that
$\widehat A^\dagger \overset{\tiny\mbox{\rm  D\!-}\ast}\leq\widehat B^\dagger$.

$ \Leftarrow$
\quad
When
$\widehat A^\dagger \overset{\tiny\mbox{\rm  D\!-}\ast}\leq\widehat B^\dagger$,
it is obvious that
$\left(\widehat A^\dagger\right)^\dagger
\overset{\tiny\mbox{\rm  D\!-}\ast}\leq
\left(\widehat B^\dagger\right)^\dagger $.
Since
$\left(\widehat A^\dagger\right)^\dagger =\widehat A$
and
$\left(\widehat B^\dagger\right)^\dagger =\widehat B$,
we get
$\widehat A \overset{\tiny\mbox{\rm  D\!-}\ast}\leq\widehat B$.
\end{proof}

If $  A\overset\ast\leq   B$, we have
 $(B-A)^\dag = B^\dag -A^\dag$
and
$(B+A)^\dag = B^\dag -\frac{1}{2}A^\dag$.
But in $\mathbb{D}^{m\times n}$,
not all dual matrices have DMPGIs.
Therefore,
in the following theorem,
we consider   properties of
$\widehat B+ \widehat A$
and
$\widehat B - \widehat A $
under the   D-star partial order.

\begin{theorem}
\label{DSPO-Sum-Th}
Let
 DMPGIs of $\widehat  A$ and $\widehat  B$   exist;
$\widehat A\overset{\tiny\mbox{\rm  D\!-}\ast}\leq\widehat B$.
Then
the DMPGIs of
$\widehat B+ \widehat A$
and
$\widehat B - \widehat A $  exist,
and
\begin{align}
\label{DSPO-Sum-1}
\left(\widehat B+\widehat A\right)^\dagger
 =
 \widehat B^\dagger- \frac{1}{2}\widehat A^\dagger,
 \
 \
\left(\widehat B-\widehat A\right)^\dagger
=
 \widehat B^\dagger-  \widehat A^\dagger.
\end{align}
\end{theorem}
\begin{proof}
Since $\widehat A\overset{\tiny\mbox{\rm  D\!-}\ast}\leq\widehat B$,
the DMPGIs of $\widehat  A $ and $\widehat  B$   exist.
Applying Theorem \ref{DSPO-Char-2-Th},
we obtain
\begin{align*}
\left\{\begin{array}{l}
\widehat B+\widehat A
=
U\begin{pmatrix}
 2{T}_1&0&0\\0&{T}_2&0\\0&0&0
\end{pmatrix}V^{  T}
+\varepsilon U\begin{pmatrix}
2A_1&A_2-{{T}_1^{-1}}{    A_4^{  T}{T}_2}&2A_3\\
2A_4-{{T}_2}{    A_2^{  T}}{{T}_1^{-1}}&    B_5&B_6\\
2A_7&B_8&0\end{pmatrix}V^{  T}
\\
\widehat B-\widehat A =U\begin{pmatrix}
0&0&0\\0&{T}_2&0\\0&0&0
\end{pmatrix}V^{  T}
+
\varepsilon U\begin{pmatrix}
0&-{{T}_1^{-1}}{    A_4^{  T}{T}_2}&0\\
-{{T}_2}{    A_2^{  T}}{{T}_1^{-1}}&    B_5&B_6\\
0&B_8&0\end{pmatrix}V^{  T}.
\end{array}\right.
\end{align*}
It follows
from Theorem \ref{lemma-2.1}
that
DMPGIs of
$\widehat B+\widehat A$
and
$\widehat B+\widehat A$
exist,
and
{\small
\begin{align*}
\left\{\begin{array}{l}
\left(\widehat B+\widehat A\right)^\dagger
= V\begin{pmatrix}
   {\displaystyle\frac12}
{T}_1^{-1}&0&0\\
0&{T}_2^{-1}&0\\
0&0&0
\end{pmatrix}U^{  T} \\
\qquad
\qquad
\qquad
+\varepsilon V
\begin{pmatrix}
-\frac12{T}_1^{-1}A_1{T}_1^{-1}&
   {\displaystyle\frac12}{T}_1^{-2}A_4^{  T}
-{T}_1^{-1}A_2{T}_2^{-1}
&   {\displaystyle\frac12}{T}_1^{-2}A_7^{  T}\\
   {\displaystyle\frac12}A_2^{  T}{T}_1^{-2}-{T}_2^{-1}A_4{T}_1^{-1}
&   -{T}_2^{-1}B_5{T}_2^{-1}
&{T}_2^{-2}B_8^{  T}\\
   {\displaystyle\frac12}A_3^{  T}{T}_1^{-2}
&{    B_6^{  T}}{   {{T}}_2^{-2}}&0
\end{pmatrix}U^{  T}
\\
\left(\widehat B-\widehat A\right)^\dagger
=
V\begin{pmatrix}
0&0&0\\
0&{T}_2^{-1}&0\\
0&0&0
\end{pmatrix}U^{  T}
+
\varepsilon V\begin{pmatrix}
0&   -{T}_1^{-1}A_2{T}_2^{-1}&0\\
   -{T}_2^{-1}A_4{T}_1^{-1}
&   -{T}_2^{-1}B_5{T}_2^{-1}
&{T}_2^{-1}B_8^{  T}\\
0&{    B_6^{  T}}{   {{T}}_2^{-1}}&0
\end{pmatrix}U^{  T}.
 \end{array}\right.
\end{align*}}
Furthermore,
applying  (\ref{3.6}) and  (\ref{3-3-9}),
we get (\ref{DSPO-Sum-1}).
\end{proof}

It is well known that if
  $A^{  T} A=A^{  T} B$ and $AA^{  T}=BA^{  T}$,
then  $A$ is below $B$ under the star partial order.
Now,
by using the method that is similar to the dual star partial order,
we introduce the T-star order.
Let $\widehat  A, \widehat  B\in\mathbb{D}^{m\times n}$,
If $\widehat  A$, $\widehat  B$ satisfy
\begin{align}
\label{DStarPartialOrder-Trans}
\widehat A^{  T}\widehat A=\widehat A^{  T}\widehat  B , \
  \widehat A\widehat A^{  T}=\widehat B\widehat A^{  T},
\end{align}
we say that $\widehat  A$ is below $\widehat  B$ under the  T-star order,
and denote it by
  $\widehat A\overset{{\tiny\mbox{\rm T-}}\ast}\leq\widehat B$.

Since any dual matrices can do transpose operation,
we suppose   $\widehat A=\varepsilon$ and $\widehat B=2\varepsilon$.
  It is easy to check that $\widehat A^T \widehat A=0=\widehat A^T \widehat B$
and
 $\widehat A \widehat A^T=0=\widehat B \widehat A^T$,
i.e.,
$\widehat A\overset{{\tiny\mbox{\rm T-}}\ast}\leq\widehat B$.
Since
 $\widehat B^T \widehat B=0=\widehat B^T \widehat A$
and
 $\widehat B \widehat B^T=0=\widehat A \widehat B^T$,
we have
$\widehat B\overset{{\tiny\mbox{\rm T-}}\ast}\leq\widehat A$.
Because
 $\widehat A \neq \widehat B$,
   T-star order is not   anti-symmetric.
Therefore,
  T-star order is not a  partial order.

Next, in the following theorem,
we suppose that  DMPGIs of dual matrices exist,
and
consider the relations between D-star partial order and T-star order.

\begin{theorem}
\label{DSPO-Char-3-Th}
Let $\widehat  A,\widehat  B\in\mathbb{D}^{m\times n}$,
and DMPGIs of $\widehat  A$ and $\widehat  B$   exist.
 We get $\widehat A\overset{\tiny\mbox{\rm  D\!-}\ast}\leq\widehat B$
if and only if
 $\widehat A\overset{{\tiny\mbox{\rm T-}}\ast}\leq\widehat B$.
\end{theorem}

\begin{proof}
``$\Leftarrow$"
Let
$\widehat  A=A+\varepsilon A_0$
and
$\widehat  B=B+\varepsilon B_0$.
It is easy to check that
 $\widehat A\overset{{\tiny\mbox{\rm T-}}\ast}\leq\widehat B$
if and only if
\begin{subnumcases}{}
\label{DSPO-Char-3-2a}
A\overset\ast\leq B
\\
\label{DSPO-Char-3-2b}
A^{  T}A_0+A_0^{  T}A
=
A^{  T}B_0+A_0^{  T}B,\
AA_0^{  T}+A_0A^{  T}=BA_0^{  T}+B_0A^{  T}.
\end {subnumcases}

Since $A\overset\ast\leq B$ and
  DMPGIs of $\widehat  A,\widehat  B$   exist,
then $A$, $B$, $A_0$ and $B_0$
can be represented in the forms as in  (\ref{2.1}),  (\ref{3.5}) and (\ref{3.7}), respectively.

Applying
$\widehat A^{  T}\widehat A=\widehat A^{  T}\widehat B$
gives
\begin{align}
\nonumber
&
V
\begin{pmatrix}
{T}_1^2&0&0\\0&0&0\\0&0&0
\end{pmatrix}V^{  T}
+\varepsilon
 V\begin{pmatrix}
 {T}_1A_1+A_1^{  T}{T}_1&{T}_1A_2&{T}_1A_3\\
     A_2^{  T}{T}_1&0&0\\
       A_3^{  T}{T}_1&0&0\end{pmatrix}V^{  T}
\\
\nonumber
&
\qquad
=
V\begin{pmatrix}
{T}_1^2&0&0\\0&0&0\\0&0&0\end{pmatrix}
V^{  T}
+\varepsilon V
\begin{pmatrix}{T}_1B_1
+A_1^{  T}{T}_1&{T}_1B_2
+A_4^{  T}{T}_2&{T}_1B_3\\
    A_2^{  T}{T}_1&0&0\\
    A_3^{  T}{T}_1&0&0
\end{pmatrix}V^{  T}.
\end{align}
Then
\begin{align}
\nonumber
\left\{\begin{array}{l}
{T}_1A_1+A_1^{  T}{T}_1
={T}_1B_1+A_1^{  T}{T}_1\\
{T}_1A_2={T}_1B_2+A_4^{  T}{T}_2\\
{T}_1A_3={T}_1B_3\end{array}\right.
\Rightarrow
\left\{\begin{array}{l}
 B_1=A_1\\
  B_2=A_2-{T}_1^{-1}A_4^{  T}{T}_2\\
   B_3=A_3.
\end{array}\right.
\end{align}
It follows from  (\ref{3.7}) that
\begin{align}
\label{DSPO-Char-3-3}
\widehat B=U\begin{pmatrix}
{T}_1&0&0\\
0&{T}_2&0\\
0&0&0
\end{pmatrix}V^{  T}
+\varepsilon
U
\begin{pmatrix}    A_1&    A_2-
{T}_1^{-1}A_4^{  T}{T}_2&    A_3\\
    B_4&B_5&B_6\\
    B_7&B_8&0
\end{pmatrix}V^{  T}.
\end{align}

Since
$\widehat A\widehat A^{  T}=\widehat B\widehat A^{  T}$
and
(\ref{DSPO-Char-3-3}),
we get
\begin{align}
\nonumber
&
U\begin{pmatrix}
{T}_1^2&0&0\\0&0&0\\0&0&0
\end{pmatrix}U^{  T}
+
\varepsilon
U\begin{pmatrix}{T}_1A_1^{  T}
+A_1{T}_1&{T}_1A_4^{  T}
&{T}_1A_7^{  T}\\
    A_4{T}_1&0&0\\
    A_7{T}_1&0&0
\end{pmatrix}U^{  T}
\\
\nonumber
&
\qquad
=U\begin{pmatrix}
{T}_1^2&0&0\\
0&0&0\\
0&0&0\end{pmatrix}U^{  T}
+
\varepsilon
U\begin{pmatrix}
{T}_1A_1^{  T}+A_1{T}_1
&{T}_1A_4^{  T}
&{T}_1A_7^{  T}\\
{T}_2A_2^{  T}+B_4{T}_1&0&0\\
    B_7{T}_1&0&0\end{pmatrix}U^{  T}.
\end{align}
Then
\begin{align}
\nonumber
\left\{\begin{array}{l}A_4{T}_1
={T}_1A_2^{  T}+B_4{T}_1\\
A_7{T}_1=B_7{T}_1
\end{array}\right.
\Rightarrow
\left\{\begin{array}{l}B_4
=A_4-{T}_2A_2^{  T}{T}_1^{-1}\\
B_7=A_7.\end{array}\right.
\end{align}
 It follows from  (\ref{DSPO-Char-3-3})  that
\begin{align}
\label{DSPO-Char-3-4}
\widehat B=U\begin{pmatrix}
{T}_1&0&0\\
0&{T}_2&0\\
0&0&0\end{pmatrix}V^{  T}
+\varepsilon U\begin{pmatrix}
A_1&    A_2-{T}_1^{-1}A_4^{  T}{T}_2&    A_3\\
 A_4-{T}_2A_2^{  T}{T}_1^{-1}&    B_5&B_6\\
  A_7&B_8&0\end{pmatrix}V^{  T}.
\end{align}
Applying (\ref{DSPO-Char-3-4})
and
Theorem \ref{DSPO-Char-2-Th},
we obtain
$\widehat A\overset{\tiny\mbox{\rm  D\!-}\ast}\leq\widehat B$.

``$\Rightarrow$"
\quad
Let
$\widehat A\overset{\tiny\mbox{\rm  D\!-}\ast}\leq\widehat B$,
by applying Theorem \ref{DSPO-Char-2-Th} ,
and it is easy to check that
$\widehat A^{  T}\widehat A=\widehat A^{  T}\widehat  B$
and
$  \widehat A\widehat A^{  T}=\widehat B\widehat A^{  T}$,
that is,
 $\widehat A\overset{{\tiny\mbox{\rm T-}}\ast}\leq\widehat B$.
\end{proof}

\section{P-Star Partial Order}
\label{Sect-4-P-Star-Partial-Order}

In Section \ref{Sect-4-P-Star-Partial-Order} ,
we introduce the   D-star order
and
show that it is a  partial order.
By using the method similar to the D-star order as in (\ref{DStarPartialOrder}),
we introduce the   P-order by using MPDGI in this section.

Let $\widehat  A=A+\varepsilon A_0$
and
$\widehat  B=B+\varepsilon B_0$.
If
\begin{align}
\label{DPSO-Char-1-Def}
\widehat A^p\widehat A=\widehat A^p\widehat B
\mbox{ \  and }
\widehat A\widehat A^p=\widehat B\widehat A^p,
\end{align}
we say that $\widehat  A$ is below $\widehat  B$ under the   P-order,
and if so,
we write
$\widehat A\overset{P}\leq\widehat B$.

It is well known that the existence of DMPGI  of dual matrix needs strict conditions ,
but the MPDGI that is closely related to DMPGI always exists for arbitrary dual matrix.
Therefore,
the  new binary relation
  is different from the D-star order.
It is meaningful to introduce   P-order
and discuss its properties and characterizations.

\begin{theorem}
\label{DPSO-Char-3-Th}
Let $\widehat  A=A+\varepsilon A_0$;
$\widehat  B=B+\varepsilon B_0$;
 DMPGIs of $\widehat  A$ and $\widehat  B$   exist,
then $\widehat A\overset P\leq\widehat B$
if and only if
\begin{align}
\label{DPSO-Char-3-Th-1}
\left\{\begin{array}{l}
A\overset\ast\leq B\\
A^\dagger A_0-A^\dagger A_0A^\dagger A
=A^\dagger B_0-A^\dagger A_0A^\dagger B \\
-AA^\dagger A_0A^\dagger+A_0A^\dagger
=-BA^\dagger A_0A^\dagger+B_0A^\dagger.
\end{array}\right.
\end{align}
\end{theorem}

\begin{proof}
Denote
$\widehat A^p=A^\dagger+\varepsilon R_p$,
where $R_p=-A^\dagger A_0A^\dagger$,
then we have
\begin{align}
\label{DPSO-Char-1-Def-1}
\left\{\begin{array}{l}
\widehat A^p\widehat A
=\left(A^\dagger+\varepsilon R_p\right)\left(A+\varepsilon A_0\right)
=A^\dagger A+\varepsilon\left(A^\dagger A_0+R_pA\right)\\
\widehat A^p\widehat B
=\left(A^\dagger+\varepsilon R_p\right)\left(B+\varepsilon B_0\right)
=A^\dagger B+\varepsilon\left(A^\dagger B_0+R_pB\right)
\end{array}\right.
\end{align}
and
\begin{align}
\label{DPSO-Char-1-Def-2}
\left\{\begin{array}{l}
\widehat A\widehat A^p
=\left(A+\varepsilon A_0\right)\left(A^\dagger+\varepsilon R_p\right)
=AA^\dagger+\varepsilon\left(AR_p+A_0A^\dagger\right)\\
\widehat B\widehat A^p
=\left(B+\varepsilon B_0\right)\left(A^\dagger+\varepsilon R_p\right)
=BA^\dagger+\varepsilon\left(BR_p+B_0A^\dagger\right).
\end{array}\right.
\end{align}
Applying
 (\ref{DPSO-Char-1-Def-1}) and  (\ref{DPSO-Char-1-Def-2}),
we get that
 (\ref{DPSO-Char-1-Def})
 is equivalent to
\begin{align}
\label{DPSO-Char-1-Def-3}
\left\{\begin{array}{l}
A^\dagger A
=A^\dagger B,\;AA^\dagger
=BA^\dagger
\\
A^\dagger A_0+R_pA
=A^\dagger B_0+R_pB,\;
AR_p+A_0A^\dagger
=BR_p+B_0A^\dagger.
\end{array}\right.
\end{align}
So, we have (\ref{DPSO-Char-3-Th-1} ).
\end{proof}

  \begin{theorem}
 \label{DPSO-Char-1-Th}
Let $\widehat A\overset{P}\leq\widehat B$.
Then
there exist orthogonal matrices $U$ and $V$
 such that
\begin{align}
 \label{DPSO-Char-1}
\left\{\begin{array}{l}\widehat A
=U\begin{pmatrix}    T_1&0&0\\0&0&0\\0&0&0\end{pmatrix}V^{  T}
+\varepsilon U\begin{pmatrix}
A_1&A_2&A_3\\A_4&A_5&A_6\\A_7&A_8&A_9\end{pmatrix}V^{  T}
\\
\widehat B
=U\begin{pmatrix}    T_1&0&0\\0&    T_2&0\\0&0&0
\end{pmatrix}V^{  T}
+
\varepsilon U\begin{pmatrix}
A_1&A_2&A_3\\A_4&    B_5&B_6\\A_7&B_8&B_9
\end{pmatrix}V^{  T},
\end{array}\right.
\end{align}
where
$T_1$ and $T_2$ are diagonal positive definite matrices.
\end{theorem}

\begin{proof}
$\Rightarrow$
\quad
Let $\widehat A\overset{P}\leq\widehat B$.
Then applying Theorem  \ref{RStar-Def} ,
we get $A$ and $B$ are of forms as in (\ref{2.1}).
Denote
\begin{align}
\label{DPSPO-Char-1-3}
 \widehat A
=U\begin{pmatrix}
  T_1&0&0\\
0&0&0\\
0&0&0
\end{pmatrix}V^{  T}
+\varepsilon U\begin{pmatrix}
A_1&A_2&A_3\\
A_4&A_5&A_6\\
A_7&A_8&A_9
\end{pmatrix}V^{  T}.
\end{align}
We have
\begin{align}
\label{DPSPO-Char-1-12}
\widehat A^p
=V\begin{pmatrix}
    T_1^{-1}&0&0\\0&0&0\\0&0&0
\end{pmatrix}U^{  T}
+\varepsilon V
\begin{pmatrix}
-{   {  T}_1^{-1}}A_1{   {  T}_1^{-1}}
&0&0\\
0&0&0\\
0&0&0
\end{pmatrix}U^{  T}.
\end{align}
 Correspondingly, partition matrix $U^{  T} B_0 V$ as follows
\begin{align}
\label{DPSPO-Char-1-13}
U^{  T} B_0 V
=
\begin{pmatrix}
B_1&B_2&B_3\\
B_4&B_5&B_6\\
B_7&B_8&B_9
\end{pmatrix}.
\end{align}

Applying(\ref{2.1}),
 (\ref{DPSPO-Char-1-3}),
 (\ref{DPSPO-Char-1-12})
 and
 (\ref{DPSPO-Char-1-13}),
we get
\begin{align}
\left\{\begin{array}{l}
\widehat A^p\widehat A
=
V\begin{pmatrix}
    I&0&0\\0&0&0\\0&0&0
\end{pmatrix}V^{  T}
+
\varepsilon V\begin{pmatrix}
0&T_1^{-1}A_2&T_1^{-1}A_3\\
0&0&0\\
0&0&0
\end{pmatrix}V^{  T}
\\
\widehat A^p\widehat B
=
V\begin{pmatrix}
    I&0&0\\
0&0&0\\
0&0&0
\end{pmatrix}V^{  T}
+
\varepsilon V\begin{pmatrix}
T_1^{-1}B_1-T_1^{-1}A_1&T_1^{-1}B_2&T_1^{-1}B_3\\
0&0&0\\
0&0&0
\end{pmatrix}V^{  T}.
\end{array}\right.
\end{align}

Since
$\widehat A^p\widehat A
=\widehat A^p\widehat B$,
we obtain
\begin{align}
\label{2.1-1}
 B_1=A_1,\
 B_2=A_2\
 \mbox{ and \ } B_3=A_3.
\end{align}

Similarly, by applying
$\widehat A\widehat A^p=\widehat B\widehat A^p$,
  $ B_1=A_1$
and
\begin{align}
\left\{\begin{array}{l}
\widehat A\widehat A^p
=U\begin{pmatrix}    I&0&0\\0&0&0\\0&0&0
\end{pmatrix}U^{  T}
+\varepsilon U\begin{pmatrix}
0&0&0\\A_4T_1^{-1}&0&0\\A_7T_1^{-1}&0&0
\end{pmatrix}U^{  T}
\\
\widehat B\widehat A^p
=
U\begin{pmatrix}
    I&0&0\\0&0&0\\0&0&0
\end{pmatrix}U^{  T}
+\varepsilon
U\begin{pmatrix}
0&0&0\\B_4T_1^{-1}&0&0\\
B_7T_1^{-1}&0&0
\end{pmatrix}U^{  T},
\end{array}\right.
\end{align}
we obtain
\begin{align}
\label{2.1-2}
 B_4=A_4\
 \mbox{ and \ }  B_7=A_7.
\end{align}
Therefore,
applying
 (\ref{DPSPO-Char-1-3}),
(\ref{2.1-1})
and
 (\ref{2.1-2})
gives (\ref{DPSO-Char-1}).
\end{proof}

\begin{example}
 \label{DPSO-Char-1-Ex}
Let
\begin{align}
\nonumber
\widehat A=\begin{pmatrix}
1&0\\
0&0
\end{pmatrix}+
\varepsilon\begin{pmatrix}
1&0\\
0&0
\end{pmatrix},\;
\widehat B=\begin{pmatrix}
1&0\\0&0
\end{pmatrix}
+
\varepsilon\begin{pmatrix}
1&0\\0&1
\end{pmatrix},
\end{align}
then
\begin{align}
\nonumber
 \widehat A^p
 =
 \widehat B^p
 =
 \begin{pmatrix}1&0\\0&0\end{pmatrix}+
 \varepsilon\begin{pmatrix}
 -1&0\\0&0\end{pmatrix}.
\end{align}
Applying
Theorem \ref{DPSO-Char-1-Th} ,
we have
$\widehat A\overset{P}\leq\widehat B$
and
$\widehat B\overset{P}\leq \widehat A$.
Since $\widehat A \neq \widehat B$,
we get that the binary relation is not anti-symmetric.
Therefore,  the P-order  is not a partial order.
 \end{example}

\begin{theorem}
 \label{DPSPO-Char-1-Th}
Let
 DMPGIs  of
$\widehat  A\in \mathbb{D}^{m\times n}$
and
$\widehat  B\in \mathbb{D}^{m\times n}$   exist.
Denote
$$\widehat A\overset{\tiny\mbox{\rm  P\!-}\ast}\leq\widehat B:
\widehat A^p\widehat A=\widehat A^p\widehat B
\mbox{ \  and }
\widehat A\widehat A^p=\widehat B\widehat A^p.$$
We call it the P-star order.
It is a partial order.

Furthermore,
there exist $U$ and $V$ that are orthogonal matrices
 such that
\begin{align}
\label{DPSPO-Char-1-1}
\left\{\begin{array}{l}
\widehat A
=
U\begin{pmatrix}
  T_1&0&0\\0&0&0\\0&0&0\end{pmatrix}
V^{  T}
+\varepsilon U
\begin{pmatrix}
A_1&A_2&A_3\\A_4&0&0\\A_7&0&0
\end{pmatrix}V^{  T}\\
\widehat B
=
U\begin{pmatrix}
  T_1&0&0\\0&  T_2&0\\0&0&0
\end{pmatrix}V^{  T}
+\varepsilon U\begin{pmatrix}
A_1&A_2&A_3\\A_4&    B_5&B_6\\A_7&B_8&0
\end{pmatrix}V^{  T},
\end{array}\right.
\end{align}
where
$T_1$ and $T_2$ are diagonal positive definite matrices.
\end{theorem}

\begin{proof}
Let $\widehat  A=A+\varepsilon A_0$,
 $\widehat  B=B+\varepsilon B_0\in \mathbb{D}^{m\times n}$.
Since
$\widehat A^p\widehat A=\widehat A^p\widehat B$,
$\widehat A\widehat A^p=\widehat B\widehat A^p$
and  DMPGIs of $\widehat  A$ and $\widehat  B$   exist,
by applying
Theorem \ref{lemma-2.1}
and
Theorem \ref{DPSO-Char-1-Th} ,
 we get  that $\widehat  A$ and $\widehat  B$
 are of the forms  as in (\ref{DPSPO-Char-1-1}).

Next, we show that P-star order satisfys reflexivity, the anti-symmetry, transitivity, respectively.

1)  Since $\widehat A\overset{P}\leq\widehat A$,  reflexivity is self-evident.

2)
Let  $\widehat  A\overset{P}\leq\widehat  B$ and
$\widehat  B\overset{P}\leq\widehat  A$.
Hence,
we get $   A\overset\ast\leq   B$ and
$   B\overset\ast\leq   A$.
Since ``$\overset\ast\leq  $" is a partial order, i.e., $A=B$.
Let  $\widehat  A$ and $\widehat  B$ be of the form as in (\ref{DPSPO-Char-1-1}).
Then ${T}_2=0$
and
\begin{align}
\label{DPSPO-Char-1-2}
\widehat B=
U\begin{pmatrix} {T}_1&0&0\\0&0&0\\0&0&0\end{pmatrix}
V^{  T}
+\varepsilon U\begin{pmatrix}
A_1&A_2 &A_3\\
A_4 &   B_5&B_6\\
A_7&B_8&0\end{pmatrix}V^{  T}.
\end{align}
Since the DMPGI of $\widehat  B$ exists,
by using Theorem \ref{lemma-2.1} ,
we have $\left( I_m-BB^{\dag} \right) B_0 \left( I_n-B^{\dag}B \right)=0$.
It follows
that
$B_5=0$,
$B_6=0$
and
$B_8=0$.
Therefore,
 $\widehat  A= \widehat  B$.
So,
the anti-symmetry holds.

3)
Next, we check the transitivity.

Since $\widehat A\overset{\tiny\mbox{\rm  P\!-}\ast}\leq\widehat B$, then
$\widehat  A$ and $\widehat  B$ are of the form as in (\ref{DPSPO-Char-1-1}).
Therefore, we obtain
\begin{align}
\label{DPSPO-Char-1-4}
\left\{\begin{array}{l}
\widehat A^p\widehat A
=
V\begin{pmatrix}
    I&0&0\\0&0&0\\0&0&0
\end{pmatrix}V^{  T}
+\varepsilon V\begin{pmatrix}
0&{T}_1^{-1}A_2&{T}_1^{-1}A_3\\0&0&0\\0&0&0
\end{pmatrix}V^{  T}
\\
\widehat B^p\widehat B
=V\begin{pmatrix}
    I&0&0\\0&    I&0\\0&0&0\end{pmatrix}V^{  T}
+\varepsilon V\begin{pmatrix}
0&0&{T}_1^{-1}A_3\\0&0&{T}_2^{-1}B_6\\0&0&0
\end{pmatrix}V^{  T},
\end{array}\right.
\end{align}
and
\begin{align}
\label{DPSPO-Char-1-5}
\left\{\begin{array}{l}
\widehat A^p\widehat A\widehat B^p
=
V\begin{pmatrix}
  {T}_1^{-1}&0&0\\0&0&0\\0&0&0
\end{pmatrix}V^{  T}
+\varepsilon V
\begin{pmatrix}
-{  {T}_1^{-1}}A_1{  {T}_1^{-1}}&0&0\\0&0&0\\0&0&0
\end{pmatrix}V^{  T}
\\
\widehat A^p\widehat A\widehat B^p\widehat B=V\begin{pmatrix}    I&0&0\\0&
0&0\\0&0&0
\end{pmatrix}V^{  T}+
\varepsilon V\begin{pmatrix}0&{T}_1^{-1}A_2&{T}_1^{-1}A_3\\0&0&0\\0&0&0
\end{pmatrix}V^{  T}.
\end{array}\right.
\end{align}
Applying (\ref{DPSPO-Char-1-4}) and (\ref{DPSPO-Char-1-5})
gives
\begin{align}
\label{DPSPO-Char-1-6}
\widehat A^p\widehat A\widehat B^p\widehat B=\widehat A^p\widehat A
 \mbox{ \ and \ }
\widehat A^p\widehat A\widehat B^p=\widehat A^p.
\end{align}

In the same way,
we get
\begin{align}
\label{DPSPO-Char-1-7}
\left\{\begin{array}{l}
\widehat A\widehat A^p
=U\begin{pmatrix}
    I&0&0\\0&0&0\\0&0&0
\end{pmatrix}U^{  T}
+\varepsilon U\begin{pmatrix}
0&0&0\\A_4T_1^{-1}&0&0\\A_7{T}_1^{-1}&0&0
\end{pmatrix}U^{  T}
\\
\widehat B\widehat B^p
=U\begin{pmatrix}
    I&0&0\\0&    I&0\\0&0&0\end{pmatrix}U^{  T}
+\varepsilon U\begin{pmatrix}
0&0&0\\A_4 {T}_1^{-1}&0&0\\A_7 {T}_1^{-1}&0&0
\end{pmatrix}U^{  T},
\end{array}\right.
\end{align}
and
\begin{align}
\label{DPSPO-Char-1-8}
\left\{\begin{array}{l}
\widehat B\widehat B^p\widehat A\widehat A^p
=U\begin{pmatrix}
    I&0&0\\0&0&0\\0&0&0\end{pmatrix}U^{  T}
+\varepsilon U\begin{pmatrix}
0&0&0\\A_4{T}_1^{-1}&0&0\\A_7{T}_1^{-1}&0&0
\end{pmatrix}U^{  T}
\\
\widehat B^p\widehat A\widehat A^p=
V\begin{pmatrix}
  {T}_1^{-1}&0&0\\0&0&0\\0&0&0
\end{pmatrix}U^{  T}
+\varepsilon V\begin{pmatrix}
-{ { {T}}_1^{-1}}A_1{ { {T}}_1^{-1}}&0&0\\0&0&0\\0&0&0
\end{pmatrix}U^{  T}.
\end{array}\right.
\end{align}
Applying (\ref{DPSPO-Char-1-7}) and (\ref{DPSPO-Char-1-8})
gives
\begin{align}
\label{DPSPO-Char-1-9}
\widehat B\widehat B^p\widehat A\widehat A^p=\widehat A\widehat A^p
\mbox{ \ and \ }
\widehat B^p\widehat A\widehat A^p=\widehat A^p.
\end{align}

Let  $\widehat B\overset{\tiny\mbox{\rm  P\!-}\ast}\leq\widehat C$,
that is, $\widehat B^p\widehat B=\widehat B^p\widehat C$ and
$\widehat B\widehat B^p=\widehat C\widehat B^p$.
Then
 \begin{align}
 \nonumber
\widehat A^p\widehat A\widehat B^p\widehat B=\widehat A^p\widehat A\widehat B^p\widehat C
\mbox{ \ and \ }
\widehat B\widehat B^p\widehat A\widehat A^p=\widehat C\widehat B^p\widehat A\widehat A^p,
\end{align}
It follows from
  (\ref{DPSPO-Char-1-6}) and (\ref{DPSPO-Char-1-9})
that
 \begin{align}
 \label{DPSPO-Char-1-10}
\widehat A^p\widehat A =\widehat A^p \widehat C
\mbox{ \ and \ }
 \widehat A\widehat A^p=\widehat C \widehat A^p,
\end{align}
that is,
  $\widehat A\overset{\tiny\mbox{\rm  P\!-}\ast}\leq\widehat C$.
  Therefore, the transitivity holds.
\end{proof}

\begin{theorem}
 \label{DPSPO-Char-2-Th}
Let $\widehat  A=A+\varepsilon A_0$;
$\widehat  B=B+\varepsilon B_0$;
DMPGIs of $\widehat  A$ and  $\widehat  B$ exist.
Then
$\widehat A\overset{\tiny\mbox{\rm  P\!-}\ast}\leq\widehat B$
if and only if
\begin{align}
\label{DPSPO-Char-1-11}
\left\{
\begin{aligned}
  A^T A& =  A^TB, \
  A A^T =B A^T
   \\
 A^TA_0&= A^TB_0, \
  A_0A^T=B_0A^T.
 \end{aligned}
 \right.
\end{align}
\end{theorem}

\begin{proof}
$\Rightarrow$
Since
 DMPGIs of $\widehat  A$ and $\widehat  B$    exist
and
$\widehat A\overset{\tiny\mbox{\rm  P\!-}\ast}\leq\widehat B$,
by using Theorem \ref{DPSPO-Char-1-Th} ,
we have
\begin{align}
\nonumber
\left\{
\begin{aligned}
  A^T A&
  =  A^TB
  =V\begin{pmatrix}T_1^{2}&0&0\\0&0&0\\0&0&0\end{pmatrix}V^T, \
  A A^T
  =B A^T
  =U\begin{pmatrix}T_1^{2}&0&0\\0&0&0\\0&0&0\end{pmatrix}U^T
   \\
 A^TA_0&
 = A^TB_0
 =V\begin{pmatrix}T_1A_1&T_1A_2&T_1A_3\\0&0&0\\0&0&0\end{pmatrix}V^T, \
  A_0A^T
  =B_0A^T
  =U\begin{pmatrix}A_1T_1&0&0\\A_4T_1&0&0\\A_7T_1&0&0\end{pmatrix}U^T.
 \end{aligned}
 \right.
\end{align}
It follows that
$A^TA_0= A^TB_0$
and
  $A_0A^T=B_0A^T$.
Therefore, we get (\ref{DPSPO-Char-1-11}).

$\Leftarrow$
Since
$ A^T A =  A^TB$
and
$  A A^T =B A^T$,
we obtain $ A\overset\ast\leq  B$.
Then applying Theorem  \ref{RStar-Def} ,
we get the forms of $A$ and $B$ as in  (\ref{2.1}).

Since  DMPGIs of $\widehat  A$ and $\widehat  B$   exist,
we write
\begin{align*}
A_0=U\begin{pmatrix}
A_1&  A_2&  A_3\\  A_4&0&0\\ A_7&0&0
\end{pmatrix}V^T,\;\;
B_0=U\begin{pmatrix}
 B_1&  B_2&  B_3\\  B_4&  B_5&  B_6\\ B_7&  B_8&0
 \end{pmatrix}V^T.
\end{align*}
Furthermore,
applying $ A^TA_0= A^TB_0$ and $ A_0A^T=B_0A^T$,
we get
$B_0=U\begin{pmatrix}
  A_1&  A_2&  A_3\\ A_4&  B_5&  B_6\\ A_7&  B_8&0
\end{pmatrix}V^T$.
By applying Theorem \ref{DPSPO-Char-1-Th},
it follows that
 $\widehat A\overset{\tiny\mbox{\rm  P\!-}\ast}\leq\widehat B$.
\end{proof}

\section{Relations among dual matrix  partial orders}
\label{Sect-5-Relationships-Star-Partial-Order}

From Theorem \ref{lemma-2.1} and Theorem \ref{MPDGI-Cha}
we see that DMPGI and MPDGI are closely related in  form.
Therefore,
D-star order induced by   DMPGI
 and
 P-star order induced by   MPDGI
 are also highly similar in form.
This is what the partial order in the real field does not have.
In this section,
we consider relationships among various types of partial orders of dual matrices.
These relations will
provide motivation for our follow-up research on matrix partial order theory
 in the real  field.

From the discussion in the above sections,
we can see that the discussion on P-star partial order and D-star  partial order
 is under the condition of the existence of DMPGI.
Therefore,
we suppose that the DMPGI of dual matrix discussed in this section exists.
Since the DMPGI of
 $\widehat A$  exists,
   the DMPGI of $\widehat A^\dag$ and $\widehat A^p$    exists.
 Therefore, we will not explain the existence of  DMPGI one by one later.


When  $\widehat  A^\dag = \widehat  A^p $,
the   D-star partial order is equivalent to the   P-star partial order.
However,in general, these two kinds of partial orders are not equivalent.
Here are some examples:
\begin{example}
 Let
 {\small
 \begin{align}
 \nonumber
\widehat A
&=
 \begin{pmatrix}
1&0&0\\0& 0&0\\0&0&0\end{pmatrix}
+\varepsilon
\begin{pmatrix}
1&1&1\\1& 0&0\\1&0&0
\end{pmatrix}, \
\widehat B
=
 \begin{pmatrix}
1&0&0\\0& 1&0\\0&0&0
\end{pmatrix}
+\varepsilon  \begin{pmatrix}
1&1&1\\1& 1&0\\1&0&0
\end{pmatrix},\
\widehat C
=
 \begin{pmatrix}
1&0&0\\0& 1&0\\0&0&0
\end{pmatrix}
+\varepsilon  \begin{pmatrix}
1&0&1\\0& 1&0\\1&0&0
\end{pmatrix}.
\end{align}}
Applying Theorem \ref{DSPO-Char-2-Th}
 and
 Theorem \ref{DPSPO-Char-1-Th},
 we have

 {\rm (1).}  $\widehat  A  \overset{\tiny\mbox{\rm  P\!-}\ast}\leq\widehat  B $,
and
  $\widehat  A$ is not below $\widehat  B$ under the partial order $\overset{\tiny\mbox{\rm  D\!-}\ast}\leq$;

 {\rm  (2).}  $\widehat  A  \overset{\tiny\mbox{\rm  D\!-}\ast}\leq\widehat  C $,
and
  $\widehat  A$ is not below $\widehat  C$ under the partial order $\overset{\tiny\mbox{\rm  P\!-}\ast}\leq$.
\end{example}

In the following theorems,
we further discuss  relationship between   D-star partial order and the   P-star partial order.
First,
 we consider characterizations of
$\widehat A\overset{\tiny\mbox{\rm  P\!-}\ast}\leq\widehat B$
when $\widehat A\overset{\tiny\mbox{\rm  D\!-}\ast}\leq\widehat B$ holds.

\begin{theorem}
\label{Two-between-PO}
Let DMPGIs of $\widehat  A=A+\varepsilon A_0$
and
$\widehat  B=B+\varepsilon B_0$
  exist;
$\widehat A\overset{\tiny\mbox{\rm  D\!-}\ast}\leq\widehat B$.
Then the following conditions are equivalent:
\begin{enumerate}
  \item[{\rm (1)}]
$\widehat A\overset{\tiny\mbox{\rm  P\!-}\ast}\leq\widehat B$;

  \item[{\rm (2)}]
There exist orthogonal matrices
$U$ and $V$
such that
{\small
\begin{align}
\label{DSPO-Char-5-1}
\left\{\begin{array}{l}\widehat A
=U\begin{pmatrix}
 T_1&0&0\\
0&0&0\\
0&0&0\end{pmatrix}V^{  T}
+\varepsilon U
\begin{pmatrix}
A_1&0&A_3\\
0&0&0\\
A_7&0&0
\end{pmatrix}V^{  T}
\\
\widehat B=
U\begin{pmatrix} T_1&0&0\\
0& T_2&0\\
0&0&0
\end{pmatrix}
V^{  T}
+\varepsilon U\begin{pmatrix}
A_1&0&A_3\\
0&  B_5&B_6\\A_7&B_8&0
  \end{pmatrix}V^{  T},
  \end{array}\right.
\end{align}}
where $T_1$ and $T_2$ are diagonal positive definite matrices;

  \item[{\rm (3)}]
$A_0^TA=A_0^TB$
and
  $ AA_0^T=B  A_0^T$;

  \item[{\rm (4)}]
  $\widehat A^\dagger\overset{\tiny\mbox{\rm  P\!-}\ast}\leq\widehat B^\dagger$.
\end{enumerate}
\end{theorem}

\begin{proof}
Since  DMPGIs of
 $\widehat  A$, $\widehat  B$   exist,
and $\widehat A\overset{\tiny\mbox{\rm  D\!-}\ast}\leq\widehat B$,
by applying Theorem \ref{DSPO-Char-2-Th},
we get the decomposition of $\widehat  A$ and $\widehat  B$
 as in
(\ref{DSPO-Char-2}).
Since DMPGIs of $\widehat  A, \widehat  B$   exist
and $\widehat A\overset{\tiny\mbox{\rm  P\!-}\ast}\leq\widehat B$,
 $\widehat  A$ and $\widehat  B$ have the forms as in  (\ref{DPSPO-Char-1-1}).

$\left(1\right)\Rightarrow\left(2\right)$:
\quad
When
$\widehat A\overset{\tiny\mbox{\rm  D\!-}\ast}\leq\widehat B$
and
$\widehat A\overset{\tiny\mbox{\rm  P\!-}\ast}\leq\widehat B$,
applying (\ref{DSPO-Char-2})
and
(\ref{DPSPO-Char-1-1}),
we get
\begin{align}
\nonumber
A_2-{{T}_1^{-1}}{    A_4^{  T}{T}_2}=A_2
\ \mbox{and}
\
A_4-{   {{T}}_2}A_2^{  T}{T}_1^{-1}=A_4.
\end{align}
Since ${T}_1$ and ${T}_2$ are invertible,
$ A_4=0$ and $A_2=0 $.
It follows from (\ref{DPSPO-Char-1-1}) that we get (2).

$\left(2\right)\Rightarrow\left(1\right)$:
\quad
When
  $\widehat A\overset{\tiny\mbox{\rm  D\!-}\ast}\leq\widehat B$
  and
   $ A_4=0$ and $A_2=0 $,
applying (\ref{DSPO-Char-2})
and
(\ref{DPSPO-Char-1-1}),
it is easy to check that
$\widehat A\overset{\tiny\mbox{\rm  P\!-}\ast}\leq\widehat B$.

$\left(2\right)\Rightarrow\left(3\right)$:
Applying (\ref{DSPO-Char-5-1})
gives
\begin{align}
\nonumber
\left\{\begin{array}{l}A_0\left(B^T-A^T\right)
=U\begin{pmatrix}
A_1&0&A_3\\0&0&0\\A_7&0&0
\end{pmatrix}
\begin{pmatrix}
0&0&0\\0&{T}_2&0\\0&0&0
\end{pmatrix}U^{  T}=0\\
\left(B^T-A^T\right)A_0
=U\begin{pmatrix}
0&0&0\\
0&{T}_2&0\\
0&0&0
\end{pmatrix}\begin{pmatrix}
A_1&0&A_3\\
0&0&0\\
A_7&0&0\end{pmatrix}U^{  T}
=0,
\end{array}\right.
\end{align}
that is,
$A_0^TA=A_0^TB$
and
  $ AA_0^T=B  A_0^T$.

 $\left(3\right)\Rightarrow\left(2\right)$:
Since the decompositions of $\widehat A$ and $\widehat B$ are as in
 (\ref{DSPO-Char-2}),
 $A_0^TA=A_0^TB$
and
  $ AA_0^T=B  A_0^T$,
we have
\begin{align}
\nonumber
\left\{\begin{array}{l}A_0\left(B^T-A^T\right)
=U\begin{pmatrix}
A_1&A_2&A_3\\A_4&0&0\\A_7&0&0
\end{pmatrix}
\begin{pmatrix}
0&0&0\\0&{T}_2&0\\0&0&0
\end{pmatrix}U^{  T}
=U\begin{pmatrix}
0&A_2{{{T}}_2}&0\\0&0&0\\0&0&0
\end{pmatrix}=0\\
\left(B^T-A^T\right)A_0
=U\begin{pmatrix}
0&0&0\\0&{T}_2&0\\0&0&0
\end{pmatrix}
\begin{pmatrix}
A_1&A_2&A_3\\A_4&0&0\\A_7&0&0
\end{pmatrix}U^{  T}
=
U\begin{pmatrix}
0&0&0\\{{ {T}}_2}A_4&0&0\\0&0&0
\end{pmatrix}U^{  T}=0.
\end{array}
\right.
\end{align}
Therefore,
$A_2=0$ and $A_4=0$.
It follows from (\ref{DSPO-Char-2}) that we get (2).

$\left(4\right)\Rightarrow\left(2\right)$
Since  DMPGIs of
 $\widehat  A$, $\widehat  B$  exist
and
$\widehat A\overset{\tiny\mbox{\rm  D\!-}\ast}\leq\widehat B$,
 $\widehat  A$ and $\widehat  B$ have the forms as in (\ref{DSPO-Char-2}).
Then  the forms of
$\widehat A^\dagger$ and $\widehat B^\dagger$
are as in
 (\ref{3.6}) and  (\ref{3-3-9}), respectively.
Applying Theorem \ref{DPSPO-Char-1-Th},
we have
\begin{align}
\nonumber
\left\{\begin{array}{l}
{T}_1^{-2}A_4^{  T}-{T}_1^{-1}A_2{T}_2^{-1}
=
{T}_1^{-2}A_4^{  T}
\\
 A_2^{  T}{T}_1^{-2}-{T}_2^{-1}A_4{T}_1^{-1}
= A_2^{  T}T_1^{-2}
\end{array}\right.
\Rightarrow
\left\{\begin{array}{l}A_2=0
\\
A_4=0\end{array}
\right.
\end{align}
It follows from (\ref{DPSPO-Char-1-1}) that we get (2).

$\left(2\right)\Rightarrow\left(4\right)$
 Applying $\widehat A\overset{\tiny\mbox{\rm  D\!-}\ast}\leq\widehat B$,
$A_2=0$, $A_4=0$,
(\ref{3.6}) and  (\ref{3-3-9}),
we get
\begin{align}
\nonumber
\left\{
\begin{aligned}
\widehat A^\dagger
&
=V\begin{pmatrix}
 T_1^{-1}&0&0
\\0&0&0
\\0&0&0\end{pmatrix}U^{  T}+\varepsilon V\begin{pmatrix}
-T_1^{-1}A_1T_1^{-1}&0& T_1^{-2}A_7^{  T}
\\   0&0&0\\   A_3^{  T}T_1^{-2}&0&0\end{pmatrix}U^{  T}
\\
\widehat B^\dagger
&=
V\begin{pmatrix}{T}_1^{-1}&0&0\\
0&{T}_2^{-1}&0\\
0&0&0\end{pmatrix}U^{  T}
+
\varepsilon V\begin{pmatrix}
-{T}_1^{-1}A_1{T}_1^{-1}&
0&
{T}_1^{-2}A_7^{  T}\\
0&
-{T}_2^{-1}B_5{T}_2^{-1}&{T}_2^{-1}B_8^{  T}\\
    A_3^{  T}{T}_1^{-2}&B_6^{  T}{T}_2^{-1}&0
\end{pmatrix}U^{  T}.
\end{aligned}
\right.
\end{align}
It follows from Theorem \ref{DPSPO-Char-1-Th}
that
$\widehat A^\dagger\overset{\tiny\mbox{\rm  P\!-}\ast}\leq\widehat B^\dagger$.
\end{proof}

 In Theorem
\ref{DSPO-Char-4-Th},
we see that
$\widehat A\overset{\tiny\mbox{\rm  D\!-}\ast}\leq\widehat B$
if and only if
$\widehat A^\dagger \overset{\tiny\mbox{\rm  D\!-}\ast}\leq\widehat B^\dagger$.
Conversely, the same is not true for the
 relationship between
  $\widehat A\overset{\tiny\mbox{\rm  P\!-}\ast} \leq\widehat B$
and
$ \widehat A^p \overset{\tiny\mbox{\rm  P\!-}\ast} \leq\widehat B^p $.
Let
\begin{align}
  \nonumber
\widehat A=\begin{pmatrix}1&0&0\\0&0&0\\0&0&0\end{pmatrix}
+\varepsilon\begin{pmatrix}1&2&3\\4&0&0\\7&0&0\end{pmatrix},\;
\widehat B=\begin{pmatrix}1&0&0\\0&2&0\\0&0&0\end{pmatrix}
+\varepsilon\begin{pmatrix}1&2&3\\4&-1&-2\\7&-3&0\end{pmatrix}.
\end{align}
Then
\begin{align}
\nonumber
 \widehat A^p=\begin{pmatrix}1&0&0\\0&0&0\\0&0&0\end{pmatrix}
 +\varepsilon\begin{pmatrix}-1&0&0\\0&0&0\\0&0&0\end{pmatrix}
\ ,
 \widehat B^p=
 \begin{pmatrix}1&0&0\\0&   \frac12&0\\0&0&0\end{pmatrix}
 +\varepsilon
 \begin{pmatrix}-1&-1&0\\-2&   \frac14&0\\0&0&0\end{pmatrix}.
\end{align}
Applying Theorem \ref{DPSPO-Char-1-Th},
we have $\widehat A\overset{\tiny\mbox{\rm  P\!-}\ast} \leq\widehat B$.

Since
\begin{align}
  \nonumber
\left(\widehat A^p\right)^p\widehat A^p
=
\begin{pmatrix}1&0&0\\0&0&0\\0&0&0
\end{pmatrix},
\;\;
\left(\widehat A^p\right)^p\widehat B^p
=
\begin{pmatrix}
1&0&0\\0&0&0\\0&0&0
\end{pmatrix}+
\varepsilon\begin{pmatrix}
0&-1&0\\0&0&0\\0&0&0
\end{pmatrix},
\end{align}
we get that
 $\left(\widehat A^p\right)^p\widehat A^p\neq \left(\widehat A^p\right)^p\widehat B^p$.
Therefore,
$\widehat A^p$ is   not below
$\widehat B^p$ under  the P-star partial order.

In the following theorem we consider characterizations of
$\widehat A^p \overset{\tiny\mbox{\rm  P\!-}\ast} \leq\widehat B^p $,
when $\widehat A\overset{\tiny\mbox{\rm  P\!-}\ast}\leq \widehat B$ holds.
\begin{theorem}
\label{DPSPO-Char-4-Th}
Let  DMPGIs of $\widehat  A=A+\varepsilon A_0$ and
$\widehat  B=B+\varepsilon B_0$
  exist,
and
$\widehat A\overset{\tiny\mbox{\rm  P\!-}\ast}\leq \widehat B$.
Then the following conditions are equivalent:
\begin{enumerate}
  \item[{\rm (1)}]
$\widehat A\overset{\tiny\mbox{\rm  D\!-}\ast}\leq\widehat B$;

  \item[{\rm (2)}]
$\widehat A^p \overset{\tiny\mbox{\rm  P\!-}\ast} \leq\widehat B^p $;

  \item[{\rm (3)}]
  $\widehat A$ and $\widehat B$ can be represented in the forms as in (\ref{DSPO-Char-5-1}).

  \item[{\rm (4)}]
$A_0^TA=A_0^TB$
and
  $ AA_0^T=B  A_0^T$.
\end{enumerate}
\end{theorem}

\begin{proof}
Since
the DMPGIs of
 $\widehat  A$ and $\widehat  B$   exist,
 and $\widehat A\overset{\tiny\mbox{\rm  P\!-}\ast}\leq \widehat B$,
by applying
Theorem \ref{DPSPO-Char-1-Th} ,
we get that   $\widehat  A$ and $\widehat  B$
have the forms as in (\ref{DPSPO-Char-1-1}).

$\left(1\right)\Rightarrow\left(3\right)$:
\quad
Since
$\widehat A\overset{\tiny\mbox{\rm  P\!-}\ast}\leq \widehat B$,
applying Theorem \ref{DSPO-Char-2-Th}
and $\widehat A\overset{\tiny\mbox{\rm  D\!-}\ast}\leq\widehat B$,
we get
$A_2-{ T_1^{-1}}{  A_4^{  T}T_2}=A_2$ and
$A_4-{ T_2}{  A_2^{  T}}{ T_1^{-1}}=A_4$.
Therefore,
we get (3).

$\left(2\right)\Rightarrow\left(3\right)$:
\quad
Applying
(\ref{DPSPO-Char-1-1}),
we have
\begin{align}
\left\{\begin{array}{l}
\widehat A^p
=V\begin{pmatrix}
T_1^{-1}&0&0\\0&0&0\\0&0&0
\end{pmatrix}U^T
+\varepsilon V
\begin{pmatrix}
-{  T_1^{-1}}A_1{  T_1^{-1}}&0&0\\0&0&0\\0&0&0
\end{pmatrix}U^T\\
\widehat B^p
=V\begin{pmatrix}  T_1^{-1}&0&0\\0&  T_2^{-1}&0\\0&0&0
\end{pmatrix}U^{\mathrm T}
+\varepsilon V\begin{pmatrix}
-{  T_1^{-1}}A_1{  T_1^{-1}}&-{  T_1^{-1}}A_2{  T_2^{-1}}&0
\\-{  T_2^{-1}}A_4{  T_1^{-1}}&-{  T_2^{-1}}B_5{  T_2^{-1}}&0
\\0&0&0
\end{pmatrix}U^{\mathrm T}.
\end{array}\right.
\end{align}

Applying
Theorem \ref{DPSPO-Char-1-Th}
gives  $\widehat A^p \overset{\tiny\mbox{\rm  P\!-}\ast} \leq\widehat B^p $,
then we obtain
$A_2=0$ and $A_4=0$.
Therefore,
applying Thereom \ref{DPSPO-Char-1-Th} gives  (3).

$\left(3\right)\Rightarrow\left(4\right)$,
$\left(3\right)\Rightarrow\left(2\right)$ and
$\left(3\right)\Rightarrow\left(1\right)$:
\quad
By applying (3),  it is easy to check that
$\widehat A\overset{\tiny\mbox{\rm  D\!-}\ast}\leq\widehat B$,
$\widehat A^p \overset{\tiny\mbox{\rm  P\!-}\ast} \leq\widehat B^p $,
$A_0^TA=A_0^TB$
and
  $ AA_0^T=B  A_0^T$.

$\left(4\right)\Rightarrow\left(3\right)$:
\quad
Since
$\widehat A\overset{\tiny\mbox{\rm  P\!-}\ast}\leq \widehat B$,
$\widehat A$ and $\widehat B$ are of the forms  as in (\ref{DPSPO-Char-1-1}).
From  $A_0^TA=A_0^TB$
and
  $ AA_0^T=B  A_0^T$,
  it follows that
 $ A_2=0$ and $A_4=0 $.
  Therefore, we get (3).
\end{proof}

Applying Theorem \ref{DPSPO-Char-2-Th},
 Theorem \ref{Two-between-PO}
and Theorem \ref{DPSPO-Char-4-Th},
we obtain Theorem \ref{DPSPO-Char-3-Th}.

\begin{theorem}
 \label{DPSPO-Char-3-Th}
Let DMPGIs of  $\widehat  A=A+\varepsilon A_0$
and
$\widehat  B=B+\varepsilon B_0$   exist.
Then the following conditions are equivalent:
\begin{enumerate}
  \item[{\rm (1)}]
$\widehat A\overset{\tiny\mbox{\rm  D\!-}\ast}\leq\widehat B$
 and
 $\widehat A \overset{\tiny\mbox{\rm  P\!-}\ast}\leq\widehat B$;

 \item[{\rm (2)}]
$\widehat A\overset{\tiny\mbox{\rm  D\!-}\ast}\leq\widehat B$
 and
$\widehat A^\dagger\overset{\tiny\mbox{\rm  P\!-}\ast}\leq\widehat B^\dagger$.

  \item[{\rm (3)}]
$\widehat A\overset{\tiny\mbox{\rm  P\!-}\ast}\leq \widehat B$
and
$\widehat A^p \overset{\tiny\mbox{\rm  P\!-}\ast} \leq\widehat B^p $;

  \item[{\rm (4)}]

  $A^T A =  A^TB$,
  $A A^T =B A^T$,
  $ A^TA_0= A^TB_0$,
  $  A_0A^T=B_0A^T$,\\
  $  A_0^TA=A_0^TB$,
  and
  $   AA_0^T=B  A_0^T $.

  \item[{\rm (5)}]
  $\widehat A$ and $\widehat B$ can be represented in the forms as in (\ref{DSPO-Char-5-1}).
\end{enumerate}
\end{theorem}

\


Let
\begin{align}
\nonumber
\widehat A=\begin{pmatrix}1&0&0\\0&0&0\\0&0&0
\end{pmatrix}+
\varepsilon\begin{pmatrix}1&2&3\\4&0&0\\7&0&0\end{pmatrix},\;
\widehat B=\begin{pmatrix}1&0&0\\0&2&0\\0&0&0\end{pmatrix}+
\varepsilon\begin{pmatrix}1&-6&3\\0&-2&-1\\7&-3&0\end{pmatrix}.
\end{align}
Applying Theorem \ref{DSPO-Char-2-Th},
it is easy to check that
$\widehat A\overset{\tiny\mbox{\rm  D\!-}\ast}\leq\widehat B$.
Since
\begin{align}
\nonumber
\widehat A^p
&=\begin{pmatrix}1&0&0\\0&0&0\\0&0&0\end{pmatrix}+
\varepsilon\begin{pmatrix}-1&0&0\\0&0&0\\0&0&0\end{pmatrix},\;
\widehat B^p=\begin{pmatrix}1&0&0\\0&   \frac12&0\\0&0&0\end{pmatrix}
+\varepsilon\begin{pmatrix}-1&3&0\\0&   \frac12&0\\0&0&0\end{pmatrix},
\\
\nonumber
\;\left(\widehat A^p\right)^\dagger\widehat A^p
&=\begin{pmatrix}
1&0&0\\0&0&0\\0&0&0
\end{pmatrix}
+\varepsilon\begin{pmatrix}
0&0&0\\0&0&0\\0&0&0
\end{pmatrix},\;\;
\left(\widehat A^p\right)^\dagger\widehat B^p
=\begin{pmatrix}1&0&0\\0&0&0\\0&0&0
\end{pmatrix}+
\varepsilon\begin{pmatrix}
0&3&0\\0&0&0\\0&0&0
\end{pmatrix},
\end{align}
we get $ \left(\widehat A^p\right)^\dagger\widehat A^p\neq\;
\left(\widehat A^p\right)^\dagger\widehat B^p$.
It follows from Theorem \ref{DSPO-Char-2-Th} that
 $\widehat A^p$ is not below $\widehat B^p$
under the D-star partial order.
Therefore,
  $\widehat A\overset{\tiny\mbox{\rm  D\!-}\ast}\leq\widehat B
\nRightarrow
 \widehat A^p\overset{\tiny\mbox{\rm  D\!-}\ast}\leq\widehat B^p$.

  \begin{theorem}
\label{DSPO-MODGI-S-3}
Let  DMPGIs of $\widehat  A=A+\varepsilon A_0$
and
$\widehat  B=B+\varepsilon B_0$
  exist.
If
  $\widehat A\overset{\tiny\mbox{\rm  D\!-}\ast}\leq\widehat B$,
then the following conditions are equivalent:
\begin{enumerate}
\item[{\rm (1)}]
$ \widehat A^p\overset{\tiny\mbox{\rm  D\!-}\ast}\leq\widehat B^p$;

\item[{\rm (2)}]
There exist orthogonal matrices $U$ and $V$
such that
{\small
\begin{align}
\label{DSPO-Char-5-3}
\left\{\begin{array}{l}\widehat A
=U\begin{pmatrix}
 T_1&0&0\\
0&0&0\\
0&0&0\end{pmatrix}V^{  T}
+\varepsilon U
\begin{pmatrix}
A_1&A_2&A_3\\
A_4&0&0\\
A_7&0&0
\end{pmatrix}V^{  T}
\\
\widehat B=
U\begin{pmatrix} T_1&0&0\\
0& T_2&0\\
0&0&0
\end{pmatrix}
V^{  T}
+\varepsilon U\begin{pmatrix}
A_1&0&A_3\\
0&  B_5&B_6\\
A_7&B_8&0
  \end{pmatrix}V^{  T},
  \end{array}
  \right.
\end{align}}
where $T_1$ and $T_2$ are diagonal positive definite matrices.

  \item[{\rm (3)}]
$BB_0^TA=AB_0^TA$
and $AB_0^TB=AB_0^TA$;
\end{enumerate}
\end{theorem}

\begin{proof}
Since
  $\widehat A\overset{\tiny\mbox{\rm  D\!-}\ast}\leq\widehat B$,
by using Theorem \ref{DSPO-Char-2-Th},
we get the  $\widehat  A$ and $\widehat  B$ are of forms as in (\ref{DSPO-Char-2}) ,
and
\begin{align}
 \label{3.3-2}
\left\{\begin{array}{l}
BB_0^TA
=U\begin{pmatrix}
{T}_1A_1^{ T}T_1&{0}&{0}\\
{T}_2^2A_4-T_2A_2^{ T}T_1&0&0\\
{0}&0&0
\end{pmatrix}V^T,\;
AB_0^TA
=U\begin{pmatrix}
{T}_1A_1^{  T}T_1&{0}&{0}\\{0}&0&0\\{0}&0&0
\end{pmatrix}V^T\\
AB_0^TB
=U\begin{pmatrix}
{T}_1A_1^{  T}T_1&  A_2T_2^2-T_1A_4^{  T}T_2&{0}\\{0}&0&0\\{0}&0&0
\end{pmatrix}V^T,\;
AB_0^TA=U\begin{pmatrix}{T}_1A_1^{ T}T_1&{0}&{0}\\{0}&0&0\\{0}&0&0
\end{pmatrix}
V^T.
\end{array}\right.
\end{align}

$\left(1\right)\Rightarrow\left(2\right)$
\quad
Applying (\ref{DSPO-Char-2}) and  Theorem \ref{lemma-2.1},
we get
\begin{align}
\nonumber
\left\{\begin{array}{l}\widehat A^p
=
V\begin{pmatrix}
{T}_1^{-1}&0&0\\0&0&0\\0&0&0
\end{pmatrix}U^T
+\varepsilon V\begin{pmatrix}
-{{T}_1^{-1}}A_1{{T}_1^{-1}}&0&0\\0&0&0\\0&0&0
\end{pmatrix}U^T
\\
\widehat B^p=
V\begin{pmatrix}
{T}_1^{-1}&0&0\\0&{T}_2^{-1}&0\\0&0&0
\end{pmatrix}U^T+
\varepsilon V\begin{pmatrix}
-{{T}_1^{-1}}A_1{{T}_1^{-1}}&
{{T}_1^{-2}}{ A_4^{ T}}-{{T}_1^{-1}}A_2{{T}_2^{-1}}&0\\
 A_2^{ T}\Sigma_1^{-2}-T_2^{-1}A_4T_1^{-1}&-{{T}_2^{-1}}B_5{{T}_2^{-1}}&0\\
 0&0&0\end{pmatrix}U^T.\end{array}
 \right.
\end{align}

Applying Theorem \ref{DSPO-Char-2-Th},
we obtain
${{T}_1^{-2}}{ A_4^{ T}}-{{T}_1^{-1}}A_2{{T}_2^{-1}}=0$
and
$A_2^{ T}T_1^{-2}-T_2^{-1}A_4T_1^{-1}=0$,
that is,
$A_2-{ T_1^{-1}}{  A_4^{  T}T_2}=0$ and $A_4-{ T_2}{  A_2^{  T}}{ T_1^{-1}}=0$.
It  froms (\ref{3.3-2})
that (\ref{DSPO-Char-5-3}) holds.

$\left(2\right)\Rightarrow\left(1\right)$:
 Applying (\ref{DSPO-Char-5-3})
gives
\begin{align}
\nonumber
\left\{\begin{array}{l}\widehat A^p
=
V\begin{pmatrix}
{T}_1^{-1}&0&0\\0&0&0\\0&0&0
\end{pmatrix}U^T
+\varepsilon V\begin{pmatrix}
-{{T}_1^{-1}}A_1{{T}_1^{-1}}&0&0\\0&0&0\\0&0&0
\end{pmatrix}U^T
\\
\widehat B^p=
V\begin{pmatrix}
{T}_1^{-1}&0&0\\0&{T}_2^{-1}&0\\0&0&0
\end{pmatrix}U^T+
\varepsilon V\begin{pmatrix}
-{{T}_1^{-1}}A_1{{T}_1^{-1}}&0&0\\
 0&-{{T}_2^{-1}}B_5{{T}_2^{-1}}&0\\
 0&0&0\end{pmatrix}U^T.\end{array}\right.
\end{align}
Then applying  Theorem \ref{DSPO-Char-2-Th}
gives
$ \widehat A^p\overset{\tiny\mbox{\rm  D\!-}\ast}\leq\widehat B^p $.

$\left(2\right)\Leftrightarrow\left(3\right)$:
 Applying (\ref{3.3-2}) ,
we obtain that
$BB_0^TA=AB_0^TA$
and $AB_0^TB=AB_0^TA$
if and only if
${T}_2^2A_4-T_2A_2^{ T}T_1=0$ and  $A_2T_2^2-T_1A_4^{  T}T_2=0$
if and only if
$A_2-{ T_1^{-1}}{  A_4^{  T}T_2}=0$ and $A_4-{ T_2}{  A_2^{  T}}{ T_1^{-1}}=0$.
Therefore,
applying Theorem \ref{DSPO-Char-2-Th} ,
we get $\left(2\right)\Leftrightarrow\left(3\right)$.
\end{proof}

Let
\begin{align}
  \nonumber
\widehat A=\begin{pmatrix}1&0&0\\0&0&0\\0&0&0\end{pmatrix}
+\varepsilon\begin{pmatrix}1&2&3\\4&0&0\\7&0&0\end{pmatrix},\;
\widehat B=\begin{pmatrix}1&0&0\\0&2&0\\0&0&0\end{pmatrix}
+\varepsilon\begin{pmatrix}1&2&3\\4&-1&-2\\7&-3&0\end{pmatrix}.
\end{align}
Applying Theorem \ref{DPSPO-Char-1-Th} ,
we can easily check that $\widehat A\overset{\tiny\mbox{\rm  P\!-}\ast} \leq\widehat B$.

Applying Theorem \ref{RStar-Def}, we have
\begin{align}
  \nonumber
\widehat A^\dagger
=\begin{pmatrix}
1&0&0\\0&0&0\\0&0&0
\end{pmatrix}
+\varepsilon
\begin{pmatrix}
-1&4&7\\2&0&0\\3&0&0
\end{pmatrix},\;
\widehat B^\dagger
=\begin{pmatrix}
1&0&0\\0&   \frac12&0\\0&0&0
\end{pmatrix}
+\varepsilon
\begin{pmatrix}
-1&3&7\\0&0&0\\3&0&0
\end{pmatrix}.
\end{align}

Since
\begin{align}
  \nonumber
\left(\widehat A^\dagger\right)^p\widehat A^\dagger
=\begin{pmatrix}
1&0&0\\0&0&0\\0&0&0
\end{pmatrix}
+
\varepsilon\begin{pmatrix}
0&4&7\\0&0&0\\0&0&0
\end{pmatrix},\;\;
\left(\widehat A^\dagger\right)^p\widehat B^\dagger
=
\begin{pmatrix}
1&0&0\\0&0&0\\0&0&0\end{pmatrix}
+\varepsilon
\begin{pmatrix}
0&3&7\\0&0&0\\0&0&0\end{pmatrix},
\end{align}
we obtain
$\left(\widehat A^\dagger\right)^p\widehat A^\dagger\neq\;
\left(\widehat A^\dagger\right)^p\widehat B^\dagger$.
Therefore,
$\widehat A^{\dag}$ is   not below $\widehat B^{\dag}$ under
the P-star partial order.
Therefore,  $\widehat A\overset{\tiny\mbox{\rm  P\!-}\ast} \leq\widehat B
\nRightarrow
 \widehat A^\dag \overset{\tiny\mbox{\rm  P\!-}\ast} \leq\widehat B^\dag $.

  \begin{theorem}
\label{DSPO-MODGI-S-4}
Let  DMPGIs of $\widehat  A=A+\varepsilon A_0$
and
$\widehat  B=B+\varepsilon B_0$
  exist.
If
  $\widehat A\overset{\tiny\mbox{\rm  P\!-}\ast} \leq\widehat B$,
then the following conditions are equivalent:
\begin{enumerate}
\item[{\rm (1)}]
$\widehat A^\dag \overset{\tiny\mbox{\rm  P\!-}\ast} \leq\widehat B^\dag$;

\item[{\rm (2)}]
There exist orthogonal matrices $U$ and $V$
such that
{\small
\begin{align}
\label{DSPO-MODGI-S-5}
\left\{\begin{array}{l}\widehat A
=U\begin{pmatrix}
 T_1&0&0\\
0&0&0\\
0&0&0\end{pmatrix}V^{  T}
+\varepsilon U
\begin{pmatrix}
A_1&A_2&A_3\\
A_4&0&0\\
A_7&0&0
\end{pmatrix}V^{  T}
\\
\widehat B=
U\begin{pmatrix} T_1&0&0\\
0& T_2&0\\
0&0&0
\end{pmatrix}
V^{  T}
+\varepsilon U\begin{pmatrix}
A_1&A_2&A_3\\
A_4&  B_5&B_6\\
A_7&B_8&0
  \end{pmatrix}V^{  T},
  \end{array}
  \right.
\end{align}}
where $T_1$ and $T_2$ are diagonal positive definite matrices,
$A_4T_1+T_2A_2^T=0$
and
$A_4^T T_2+T_1 A_2 =0$.
\end{enumerate}
\end{theorem}

\begin{proof}
Since
the DMPGIs of
 $\widehat  A$ and $\widehat  B$  exist,
and $\widehat A\overset{\tiny\mbox{\rm  P\!-}\ast}\leq \widehat B$,
by using
Theorem \ref{DPSPO-Char-1-Th}
we get that  $\widehat  A$ and $\widehat  B$
are as in (\ref{DPSPO-Char-1-1}),
$\widehat A^\dagger$ is of the form as in (\ref{3.6}),
and
\begin{align}
\label{DPSPO-Char-1-14}
\widehat B^\dagger
=
U\begin{pmatrix}
  T_1^{-1}&0&0\\0&  T_2^{-1}&0\\0&0&0
\end{pmatrix}V^{  T}
+\varepsilon U\begin{pmatrix}
-T_1^{-1}A_1T_1^{-1}&-T_1^{-1}A_2T_2^{-1}& T_1^{-2}A_7^T\\
-T_2^{-1}A_4T_1^{-1}&-T_2^{-1}B_5T_2^{-1}&T_2^{-2} B_8^T\\
A_3^T T_1^{-2}&B_6^TT_2^{-2} &0
\end{pmatrix}V^{  T}.
\end{align}

$\left(1\right)\Rightarrow\left(2\right)$
\quad
If
$\widehat A^\dag \overset{\tiny\mbox{\rm  P\!-}\ast} \leq\widehat B^\dag$,
then
applying Theorem \ref{DPSPO-Char-1-Th} ,
we get
$T_1^{-2}A_4^{  T}+T_1^{-1}A_2T_2^{-1}=0$
and
$A_2^{  T}T_1^{-2}+T_2^{-1}A_4T_1^{-1}=0$.
Therefore,
we get (2).

$\left(2\right)\Rightarrow\left(1\right)$
\quad
If $\widehat A$ and $\widehat B$ are given as in (\ref{DSPO-MODGI-S-5}),
it is easy to check that  (1) holds.
\end{proof}


\section*{Disclosure statement}
No potential conflict of interest was reported by the authors.

\section*{Funding}
This work was supported partially by the
  Guangxi Natural Science Foundation [grant number 2018GXNSFDA281023],
National Natural Science Foundation of China [grant number 12061015],
Xiangsihu Young Scholars Innovative Research Team of Guangxi Minzu University
 [grant number 2019RSCXSHQN03]
and
Thousands of Young and Middle-aged Key Teachers Training Programme in Guangxi Colleges and Universities [grant number GUIJIAOSHIFAN2019-81HAO].

\end{document}